\documentclass[a4paper]{article}

\usepackage{arydshln}
\usepackage{graphicx}
\usepackage{amsmath}
\usepackage{amssymb, verbatim}
\usepackage{mathrsfs}
\usepackage{dsfont}
\usepackage{url}
\usepackage[usenames,dvipsnames,svgnames,table]{xcolor}
\usepackage[compress]{cite}
\usepackage{mdwlist}
\usepackage{paralist}
\usepackage{setspace}

\newtheorem{lemma}{\sc Lemma}
\newtheorem{theorem}[lemma]{\sc Theorem}

\newtheorem{corollary}[lemma]{\sc Corollary}

\newtheorem{remark}{\sc Remark}
\newtheorem{example}{\sc Example}
\newtheorem{assumption}{\sc Assumption}
\newtheorem{definition}{\sc Definition}

\renewcommand{\matrix}[2]{\left[\begin{array}{#1} #2 \end{array}\right] }
\newcommand{\T}{\top}

\DeclareMathOperator*{\arginf}{arg\,inf}
\newcommand{\vect}{\mathrm{vec}}
\DeclareMathOperator*{\diag}{diag}

\newcommand{\newqed}{$\rhd$}

\newcommand*{\LONGVERSION}{}

\newenvironment{proof}{{\noindent \bf Proof:\ }}{ \hfill $\square$}

\date{}

\title{Optimal Control Design under Limited Model \\ Information for Discrete-Time Linear Systems with \\ Stochastically-Varying Parameters\thanks{An early version of this paper was presented at the 51st IEEE Conference on Decision and Control, 2012~\cite{FJCDC2012}. The work was supported by the Swedish Research Council and the Knut and Alice Wallenberg Foundation.}}

\author{Farhad~Farokhi$^\dag$~and~Karl~H.~Johansson\thanks{F.~Farokhi and K.~H.~Johansson are with ACCESS Linnaeus Center, School of Electrical Engineering, KTH Royal Institute of Technology, SE-100 44 Stockholm, Sweden. E-mails:\{farokhi,kallej\}@ee.kth.se }}

\begin{document}

\maketitle

\begin{abstract} The value of plant model information available in the control design process is discussed. We design optimal state-feedback controllers for interconnected discrete-time linear systems with stochas-tically-varying parameters. The parameters are assumed to be independently and identically distributed random variables in time. The design of each controller relies only on (\emph{i})~exact local plant model information and (\emph{ii})~statistical beliefs about the model of the rest of the system. We consider both finite-horizon and infinite-horizon quadratic cost functions. The optimal state-feedback controller is derived in both cases. The optimal controller is shown to be linear in the state and to depend on the model parameters and their  statistics in a particular way. Furthermore, we study the value of model information in optimal control design using the performance degradation ratio which is defined as the supremum (over all possible initial conditions) of the ratio of the cost of the optimal controller with limited model information scaled by the cost of the optimal controller with full model information. An upper bound for the performance degradation ratio is presented for the case of fully-actuated subsystems. Comparisons are made between designs based on limited, statistical, and full model information. Throughout the paper, we use a power network example to illustrate concepts and results.
\end{abstract}

\section{Introduction}
\subsection{Motivation}
Large-scale systems such as automated highways~\cite{swaroop:462,Horowitz871301}, aircraft and satellite formations~\cite{kapila2000spacecraft,Giulietti887447}, supply chains~\cite{Braun2003229,Dunbar2007}, power grids and other shared infrastructures~\cite{Massoud2005,Negenborn2010} are typically composed of several locally controlled subsystems that are connected to each other either through the physical dynamics, the communication infrastructure, or the closed-loop performance criterion. The problem of designing these local controllers, widely known as distributed or decentralized control design, is an old and well-studied problem in the literature~\cite{witsenhausen1968counterexample,levine1971optimal,wang1973stabilization,blondel1997np,rotkowitz2006characterization,voulgaris2003optimal,Siljak1991decentralized}.
Although the controller itself is highly structured for these large-scale systems, it is commonly assumed that the complete model of the system is available and the design is done in a centralized fashion using the global plant model information. However, this assumption is usually not easily satisfied in practice. For instance, this might be because the design of each local controller is done by a separate designer with no access to the global plant model because the full plant model information is not available at the time of design or it might change later. Recently, this concern has become more important as engineers implement large-scale systems using off-the-shelf components which are designed in advance with limited prior knowledge of their future operating condition. Another reason to consider control design based on only local information is to simplify the tuning and the maintenance of the system. For instance, dependencies between cyber components in a large system can cause complex interactions influencing the physical plant, not present without the controller. Privacy concerns could also be a motivation for designing control actions using only local information. For further motivations behind optimal control design using local model information, see~\cite{Farokhi-thesis2012}.

As an illustrative physical example, let us consider a power network control problem with power being generated in generators and distributed throughout the network via transmission lines (e.g.,~\cite{kundur1994power,anderson2003power}). It is fairly common to assume that the power consumption of the loads in such a network can be modeled stochastically with a priori known statistics, such as, mean and variance extracted from long term observations~\cite{Loparo1985,Brucoli19909,Wu1986}. When the load variations are ``small enough'', local generators meet these demand variations. These variations shift the generators operating points, and consequently, change their model parameters. If the loads are modeled as impedances, they change the system model by changing the transmission line impedances.  As power networks are typically implemented over a vast geographical area, it is inefficient or even impossible to gather all these model information variations or to identify all the parameters globally. Even if we could gather all the information and identify the whole system based on them, it might take very long and by then the information might be outdated (noting that the model parameters vary stochastically over time). This motivates the interest in designing local controllers for these systems based on only local model information and statistical model information of the rest of the system. We revisit this power network problem in detail for a small example in the paper. A recurring example is used to explain the underlying definitions as well as the mathematical results. It is not difficult to see that similar examples can also be derived for process control, intelligent transportation, irrigation systems, and other shared infrastructures.

\subsection{Related Studies}
Optimal control design under limited model information has recently attracted attention. The authors in~\cite{langbort2010distributed} introduced control design strategies as mappings from the set of plants of interest to the set of eligible controllers. They studied the quality of these control design strategies using a performance metric called the competitive ratio; i.e., the worst case ratio of the closed-loop performance of a given control design strategy to the closed-loop performance of the optimal control design with full model information. Clearly, the smaller the competitive ratio is, the more desirable the control design strategy becomes since it can closely replicate the performance of the optimal control design strategy with full model information while only relying on local plant model information. They showed that for discrete-time systems composed of scalar subsystems, the deadbeat control design strategy is a minimizer of the competitive ratio. Additionally, the deadbeat control design strategy is undominated; i.e., there is no other control design strategy that performs always better while having the same competitive ratio. This work was later generalized to limited model information control design methods for inter-connected linear time-invariant systems of arbitrary order in~\cite{FLJ2012}. In that study, the authors investigated the best closed-loop performance that is achievable by structured static state-feedback controllers based on limited model information. It was shown that the result depends on the subsystems interconnection pattern and availability of state measurements. Whenever there is no subsystem that cannot affect any other subsystem and each controller has access to at least the state measurements of its neighbors, the deadbeat strategy is the best limited model information control design method. However, the deadbeat control design strategy is dominated (i.e., there exists another control design strategy that outperforms it while having the same competitive ratio) when there is a subsystem that cannot affect any other subsystem. These results were generalized to structured dynamic controllers when the closed-loop performance criterion is set to be the $H_2$-norm of the closed-loop transfer function~\cite{farokhi2011dynamic}. In this case, the optimal control design strategy with limited model information is static even though the optimal structured state-feedback controller with full model information is dynamic~\cite{Swigart5717538,Shah5718116}. Later in~\cite{FarokhiSIAM13}, the design of dynamic controllers for optimal disturbance accommodation was discussed. It was shown that in some cases an observer-based-controller is the optimal architecture also under limited model information. Finally, in~\cite{farokhi2013optimal}, it was shown that using an adaptive control design strategy, the designer can achieve a competitive ratio equal to one when the considered plant model belongs to a compact set of linear time-invariant systems and the closed-loop performance measure is the ergodic mean of a quadratic function of the state and control input (which is a natural extension of the $H_2$-norm of the closed-loop system considering that the closed-loop system in this case is nonlinear due to the adaptive controller).

In all these studies, the model information of other subsystems are assumed to be completely unknown which typically results in conservative controllers because it forces the designer to study the worst-case behavior of the control design methods. In this paper, we take a new approach by assuming that a statistical model is available for the parameters of the other subsystems. There have been many studies of optimal control design for linear discrete-time systems with stochastically-varying parameters~\cite{Koning1982,DeKoning1984,aoki1967optimization,doi:10.1080/00207178408933286,Imer2006}. In these papers, the optimal controller is typically calculated as a function of model parameter statistics. Considering a different problem formulation, in this paper, we assume each controller design is done using the exact model information of its corresponding subsystem and the other subsystems' model statistics.

Note that studying the worst-case behavior of the system using the competitive ratio is not the only approach for optimal control design under limited model information. For instance, the authors in~\cite{ando1963near,sezer1986nested,sethi1998near} developed methods for designing near-optimal controllers using only local model information whenever the coupling between the subsystems is negligible. However, not even the closed-loop stability can be guaranteed when the coupling grows. As a different approach, in a recent study~\cite{DerooCDC2012}, the authors used an iterative numerical optimization algorithm to solve a finite-horizon linear quadratic problem in a distributed way using only local model information and communication with neighbors. However, this approach (and similarly~\cite{Mrtensson5400233,dunbar2007distributed}) require many rounds of communication between the subsystems to converge to a reasonable neighborhood of the optimal controller. To the best of our knowledge, there is also no stopping criteria (for terminating the numerical optimization algorithm) that uses only local information. There have been some studies in developing stopping criteria but these studies require global knowledge of the system~\cite{Giselsson5717026,farokhi2012distributed}.
Recently, there has been an attempt for designing optimal controllers using only local model information for linear systems with stochastically-varying parameters~\cite{MishraCDC2012}. However, that setup is completely different from the problem that is considered in this paper. First, the authors of~\cite{MishraCDC2012} considered the case where the $B$-matrix was parameterized with stochastic variables but in our setup the $A$-matrix is assumed to be stochastic. Additionally, in~\cite{MishraCDC2012}, the infinite-horizon problem was only considered for the case of two subsystems, while here we present all the results for arbitrary number of subsystems. In this paper, we introduce the concept of performance degradation ratio as a measure to study the value information in optimal control design. Furthermore, the proof techniques are different since the authors of~\cite{MishraCDC2012} use a team-theoretic approach to solve the problem opposed to the approach presented in this paper.

\subsection{Main Contribution}
The main contribution of this paper is to study the value of plant model information available in the control design process. To do so, we consider limited model information control design for discrete-time linear systems with stochastically-varying parameters. First, in Theorem~\ref{tho:1}, we design the optimal finite-horizon controller based on exact local model information and global model parameter statistics. We generalize these results to infinite-horizon cost functions in Theorem~\ref{tho:2} assuming that the underlying system is mean square stabilizable; i.e., there exists a constant matrix that can mean square stabilizes the system~\cite{Koning1982}. However, in Corollary~\ref{cor:1}, we partially relax the assumptions of Theorem~\ref{tho:2} to calculate the infinite-horizon optimal controller whenever the underlying system is mean square stabilizable under limited model information. This new concept is defined through borrowing the idea of control design strategies from~\cite{langbort2010distributed,FLJ2012}. We define a special class of control design strategies to construct time-varying control gains for each subsystem. We say that a system is mean square stabilizable under limited model information if the intersection of this special class of control design strategies (that use only local model information) and the set of mean square stabilizing control design strategies is nonempty; i.e., there exists a control design strategy that uses only local model information and it can mean square stabilizes the system (see Definition~\ref{def:2} for more details).

Using the closed-loop performance of the optimal controller with limited model information, we study the effect of lack of full model information on the closed-loop performance. Specifically, we study the ratio of the cost of the optimal control design strategy with limited model information scaled by the cost of the optimal control design strategy with full model information (which is introduced in Theorems~\ref{tho:3} and~\ref{tho:4} for finite-horizon and infinite-horizon cost functions, respectively). We call the supremum of this ratio over the set of all initial conditions, the performance degradation ratio. In Theorem~\ref{tho:upper}, we find an upper bound for the performance degradation ratio assuming the underlying systems are fully-actuated (i.e., they have the same number of inputs as the state dimension). As a future direction for research, one might be able to generalize these results to designing structured state-feedback controllers following the same line of reasoning as in~\cite{swigart2010explicit}.

An early and brief version of the paper was presented as~\cite{FJCDC2012}. The current paper is a considerable extension of~\cite{FJCDC2012} as the results have been generalized, a new literature survey has been included, and a power network example has been introduced to  illustrate concepts and results throughout the paper.

\subsection{Paper Outline}
The rest of the paper is organized as follows. We start with introducing the system model in Section~\ref{sec:modelling}. In Section~\ref{sec:Optimal Control Design}, we design optimal controller for each subsystem based on limited model information (i.e., using its own model information and the statistical belief about the other subsystems). We start by the finite-horizon optimal control problem and then generalize the results to infinite-horizon cost functions. In Section~\ref{sec:fullmodelinfo}, we introduce the optimal controller for both finite-horizon and infinite-horizon cost functions when using the full model information. In Section~\ref{sec:Competitive Ratio}, we study the value of plant model information in optimal control design using the performance degradation ratio. Finally, the conclusions and directions for future research are presented in Section~\ref{sec:con}.

\subsection{Notation}
The sets of integers and reals are denoted by $\mathbb{Z}$ and $\mathbb{R}$, respectively. We denote all other sets with calligraphic letters such as $\mathcal{A}$ and $\mathcal{X}$. Specifically, we define $\mathcal{S}_{++}^n$ ($\mathcal{S}_{+}^n$) as the set of all symmetric matrices in $\mathbb{R}^{n\times n}$ that are positive definite (positive semidefinite). Matrices are denoted by capital roman letters such as $A$. We use the notation $A_{ij}$ to denote a submatrix of matrix $A$ (its dimension and position will be defined in the text). The entry in the $i^{\textrm{th}}$ row and the $j^{\textrm{th}}$ column of the matrix $A$ is denoted $a_{ij}$.  We define $A > (\geq) 0$ as $A\in \mathcal{S}_{++}^n(\mathcal{S}_{+}^n)$ and $A > (\geq) B$ as $A-B > (\geq) 0$. Let $A\otimes B\in\mathbb{R}^{np\times qm}$ denote the Kronecker product between matrices $A\in\mathbb{R}^{n\times m}$ and $B\in\mathbb{R}^{p\times q}$; i.e.,
$$
A\otimes B=\matrix{ccc}{ a_{11}B & \cdots & a_{1m}B \\ \vdots & \ddots & \vdots \\ a_{n1}B & \cdots & a_{nm}B }.
$$
For any positive integers $n$ and $m$, we define the mapping $\vect:\mathbb{R}^{n\times m} \rightarrow \mathbb{R}^{nm}$ as $\vect(A)=[A_{1}^\T\;A_{2}^\T\; \cdots A_{m}^\T]^\T$
where $A_{i}$, $1\leq i\leq m$, denotes the $i^{\textrm{th}}$ column of $A$. The mapping $\vect^{-1}:\mathbb{R}^{nm}\rightarrow \mathbb{R}^{n\times m}$ is the inverse of $\vect(\cdot)$, where the dimension of the matrix will be clear from the context. It is useful to note that both $\vect$ and $\vect^{-1}$ are linear operators.  Finally, for any given positive integers $n$ and $m$, we define the discrete Riccati operator $\mathbf{R}:\mathbb{R}^{n\times n}\times \mathcal{S}_{+}^{n} \times \mathbb{R}^{n\times m} \times \mathcal{S}_{++}^{m} \rightarrow \mathcal{S}_{+}^{n}$ as
$\mathbf{R}(A,P,B,R)=A^\T (P-P B (R+B^\T P B)^{-1} B^\T P) A$ for any $A\in\mathbb{R}^{n\times n}$, $P\in\mathcal{S}_{+}^{n}$, $B\in\mathbb{R}^{n\times m}$, and $R\in\mathcal{S}_{++}^{m}$.

\section{Control Systems with Stochastically-Varying Parameters} \label{sec:modelling}
Consider a discrete-time linear system with stochastically-varying parameters composed of $N$~subsystems with each subsystem represented in state-space form as
\begin{equation} \label{eqn:eachsubsystem}
x_i(k+1)=\sum_{j=1}^N A_{ij}(k) x_j(k)+B_{ii}(k) u_i(k),
\end{equation}
where $x_i(k)\in\mathbb{R}^{n_i}$ and $u_i(k)\in\mathbb{R}^{m_i}$ denote subsystem~$i$, $1\leq i\leq N$, state vector and control input, respectively.

\begin{remark} Linear systems with stochastically-varying parameters have been studied in many applications including power networks~\cite{Loparo1985,Brucoli19909}, process control~\cite{doi:10.1080/00207178908559618}, finance~\cite{Dombrovskii2003Lyashenko}, and networked control~\cite{schenato2007foundations,Imer2006}. Various system theoretic properties and control design methods have been developed for these systems~\cite{Koning1982,DeKoning1984,aoki1967optimization,doi:10.1080/00207178408933286}.
\end{remark}

We make the following two standing assumptions:

\begin{assumption} \label{assmp:iid} The submatrices $A_{ij}(k)$, $1\leq i,j\leq N$, are independently distributed random variables in time; i.e., $\mathbb{P}\{A_{ij}(k_1) \in\mathcal{X}\,|\,A_{ij}(k_2)\}=\mathbb{P}\{A_{ij}(k_1)\in\mathcal{X}\}$ for any $\mathcal{X}\subseteq\mathbb{R}^{n_i\times n_j}$ whenever~$k_1\neq k_2$.
\end{assumption}

\begin{assumption} \label{assum:subsystemindependence} The subsystems are statistically independent of each other; i.e., $\mathbb{P}\{A_{i j}(k)\in\mathcal{X} \,|\, A_{i' j'}(k)\}\linebreak =\mathbb{P}\{A_{i j}(k)\in\mathcal{X}\}$ for any $\mathcal{X}\subseteq\mathbb{R}^{n_i\times n_j}$ and $1\leq j,j' \leq N$ whenever $i\neq i'$.
\end{assumption}

We illustrate these properties on a small power network example. We will frequently revisit this example to demonstrate the developed results as well as their implications.

\begin{figure}
\centering
\includegraphics[width=0.3\linewidth]{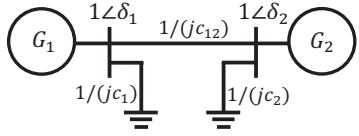}
\caption{\label{figure1} Schematic diagram of the power network in Example~\ref{example:1}. }
\end{figure}

\setcounter{example}{0}
\begin{example} \label{example:1} 
Let us consider the power network composed of two generators shown in Figure~\ref{figure1} from~\cite[pp.\;64--65]{Ghandhari00}, see also~\cite{kundur1994power}. We can model this power network as
\begin{equation*}
\begin{split}
\dot{\delta}_1(t)&=\omega_1(t), \\
\dot{\omega}_1(t)&=\frac{1}{M_1}\big[P_{1}(t)-c_{12}^{-1} \sin(\delta_1(t)-\delta_2(t)) -c_{1}^{-1}\sin(\delta_1(t))-D_1\omega_1(t) \big], \\
\dot{\delta}_2(t)&=\omega_2(t), \\
\dot{\omega}_2(t)&=\frac{1}{M_2}\big[P_{2}(t)-c_{12}^{-1}\sin(\delta_2(t)-\delta_1(t)) -c_{2}^{-1}\sin(\delta_2(t))-D_2\omega_2(t) \big],
\end{split}
\end{equation*}
where $\delta_i(t)$, $\omega_i(t)$, and $P_{i}(t)$ are, respectively, the phase angle of the terminal voltage of  generator~$i$,  its rotation frequency, and its input mechanical power. We assume that $P_1(t)=1.6+v_1(t)$ and $P_2(t)=1.2+v_2(t)$, where $v_1(t)$ and $v_2(t)$ are the continuous-time control inputs of this system. The power network parameters can be found in Table~\ref{table:1} (see~\cite{Ghandhari00,kundur1994power} and references therein for a discussion on these parameters). Now, we can find the equilibrium point $(\delta_1^*,\delta_2^*)$ of this system and linearize it around this equilibrium. Let us discretize the linearized system by applying Euler's constant step scheme with sampling time $\Delta T=300\,\mathrm{ms}$, which results in
\begin{equation*}
\begin{split}
\hspace{-.04in}\matrix{c}{ \hspace{-.08in} \Delta\delta_1(k+1) \hspace{-.08in} \\ \hspace{-.08in} \Delta\omega_1(k+1) \hspace{-.08in} \\ \hspace{-.08in} \Delta\delta_2(k+1) \hspace{-.08in} \\ \hspace{-.08in} \Delta\omega_2(k+1) \hspace{-.08in}} \hspace{-.06in}&=\hspace{-.06in}\matrix{cccc}{1 & \Delta T & 0 & 0 \\ \hspace{-.1in}\xi_1\hspace{-.06in} & 1-\frac{\Delta TD_1}{M_1} & \frac{\Delta T\cos(\delta_1^*-\delta_2^*)}{c_{12}M_1} & 0 \\ 0 & 0 & 1 & \Delta T \\ \frac{\Delta T\cos(\delta_2^*-\delta_1^*)}{c_{12}M_2} & 0 & \hspace{-.1in}\xi_2 & \hspace{-.06in}1-\frac{\Delta TD_2}{M_2}}\hspace{-.08in}
\matrix{c}{ \hspace{-.08in} \Delta\delta_1(k) \hspace{-.08in} \\ \hspace{-.08in} \Delta\omega_1(k) \hspace{-.08in} \\ \hspace{-.08in} \Delta\delta_2(k) \hspace{-.08in} \\ \hspace{-.08in} \Delta\omega_2(k) \hspace{-.08in}}\hspace{-.06in}+\hspace{-.06in}\matrix{cc}{0 \\ \hspace{-.08in}u_1(k)\hspace{-.08in} \\ 0  \\ \hspace{-.08in}u_2(k)\hspace{-.08in}}\hspace{-.05in},
\end{split}
\end{equation*}
with $\xi_1\hspace{-.03in}=\hspace{-.03in}-\Delta T( c_{12}^{-1}\hspace{-.03in}\cos(\delta_1^*\hspace{-.03in}-\hspace{-.03in}\delta_2^*)+c_1^{-1}\cos(\delta_1^*))/M_1$ and $\xi_2\hspace{-.03in}=\hspace{-.03in}-\Delta T(c_{12}^{-1}\hspace{-.03in}\cos(\delta_2^*\hspace{-.03in}-\hspace{-.03in}\delta_1^*)+c_2^{-1}\cos(\delta_2^*))/M_2$, where $\Delta\delta_1(k)$, $\Delta\delta_2(k)$, $\Delta\omega_1(k)$, and $\Delta\omega_2(k)$ denote the deviation of $\delta_1(t)$, $\delta_2(t)$, $\omega_1(t)$, and $\omega_2(t)$ from their equilibrium points at time instances $t=k\Delta T$. Additionally, let the actuators be equipped with a zero order hold unit which corresponds to $v_i(t)=u_i(k)$ for all $k\Delta T\leq t<(k+1)\Delta T$. Let us assume that we have connected impedance loads to each generator locally, such that the parameters $c_1$ and $c_2$ vary stochastically over time according to the load profiles. Furthermore, assume that each generator changes its input mechanical power according to these local load variations (to meet their demand and avoid power shortage). Doing so, we would not change the equilibrium point $(\delta_1^*,\delta_2^*)$. For this setup, we can model the system as a discrete-time linear system with stochastically-varying parameters 
$$
x(k+1)=A(k)x(k)+Bu(k), 
$$
where
$$
x(k)=\matrix{c}{ \Delta\delta_1(k) \\ \Delta\omega_1(k) \\ \Delta\delta_2(k) \\ \Delta\omega_2(k) }\hspace{-.05in},
\hspace{.4in} u(k)=\matrix{cc}{ u_1(k) \\ u_2(k) }\hspace{-.05in}, \hspace{.4in}
B=\matrix{cccc}{0 & 0 \\ 1 & 0 \\ 0 & 0 \\ 0 & 1}\hspace{-.05in},
$$
and
$$
A(k)=\matrix{cccc}{
    1.0000  &  0.3000 &        0  &       0\\
  -45.6923-6.9297\alpha_1(k)  &  0.9250 &  29.3953  &       0\\
         0  &       0 &   1.0000  &  0.3000\\
   23.5163  &       0 & -37.3757 -8.1485\alpha_2(k)  &  0.9400}\hspace{-.05in},
$$
where $\alpha_i(k)$, $i=1,2$, denotes the deviation of the admittance $c_i^{-1}$ from its nominal value in Table~\ref{table:1}. Let us assume that $\alpha_1(k)$ and $\alpha_2(k)$ are independently and identically distributed random variables in time with $\alpha_1(k)\sim\mathcal{N}(0,0.1)$ and $\alpha_2(k)\sim\mathcal{N}(0,0.3)$. Note that in this example, $\alpha_i(k)$ is a stochastically-varying parameter of subsystem~$i$ describing the dynamics of the local power consumption. It only appears in the model of subsystem~$i$; i.e., in $\{A_{ij}(k)|1\leq j\leq N\}$. In the rest of the paper when discussing this example and for designing controller~$i$, we assume that we only have access to the exact realization of $\alpha_i(k)$ in addition to the statistics of the other subsystem. This is motivated by the fact that the controller of the other generator might not have access to this model information. 
\hfill\newqed
\end{example}

\renewcommand{\arraystretch}{1.3}
\begin{table}
\footnotesize
\caption{Nominal values of power system parameters in Example~\ref{example:1}. }
\label{table:1}
\centering
\begin{tabular}{|c||c|c|c|c|c|c|c|} \hline
Parameters            & $M_1$                & $M_2$ & $c_{12}$ & $c_1$ & $c_2$ & $D_1$ & $D_2$ \\ \hline \hline Nominal Value (p.u.) & $2.6 \times 10^{-2}$ & $3.2 \times 10^{-2}$ & $0.40$ & $0.50 $ & $0.50$ &  $6.4\times 10^{-3}$ & $6.4\times 10^{-3}$ \\ \hline
\end{tabular}
\end{table}

We define the concatenated system from~(\ref{eqn:eachsubsystem}) as 
\begin{equation} \label{eqn:sys}
x(k+1)=A(k)x(k)+B(k)u(k), 
\end{equation}
where
$
x(k)=[x_1(k)^\top \; \cdots \; x_N(k)^\top]^\top \in\mathbb{R}^{n}$ and $u(k)=[u_1(k)^\top \; \cdots \; u_N(k)^\top]^\top \in\mathbb{R}^{m},
$
with $n=\sum_{i=1}^N n_i$ and $m=\sum_{i=1}^N m_i$. Let $x_0=x(0)$. We also use the notations $\bar{A}_{ij}(k)=\mathbb{E}\{A_{ij}(k)\}$, $\tilde{A}_{ij}(k)=A_{ij}(k)-\bar{A}_{ij}(k)$, $\bar{A}(k)=\mathbb{E}\{A(k)\}$, and $\tilde{A}(k)=A(k)-\bar{A}(k)$. Furthermore, for all $1\leq i\leq N$, we introduce the notations
$$
B_{i}(k)=\left[\hspace{-.06in}\begin{array}{c}0_{(\sum_{j=1}^{i-1} n_j) \times m_i} \\ B_{ii}(k) \\ 0_{(\sum_{j=i+1}^{N} n_j) \times m_i} \end{array}\hspace{-.06in}\right]\hspace{-.05in}, \hspace{0.1in} \tilde{A}_i(k)=\left[\hspace{-.06in}\begin{array}{ccc}0_{(\sum_{j=1}^{i-1} n_j) \times n_1} & \cdots & 0_{(\sum_{j=1}^{i-1} n_j) \times n_N}\\ \tilde{A}_{i1}(k) & \cdots & \tilde{A}_{iN}(k) \\ 0_{(\sum_{j=i+1}^{N} n_j) \times n_1} & \cdots & 0_{(\sum_{j=i+1}^{N} n_j) \times n_N} \end{array}\hspace{-.06in}\right]\hspace{-.05in}.
$$
Now, we are ready to calculate the optimal controller under model information constraints.

\section{Optimal Control Design with Limited Model Information} \label{sec:Optimal Control Design}
In this section, we study the finite-horizon and infinite-horizon optimal control design using exact local model information and statistical beliefs about other subsystems. We consider state-feedback control laws $u_i(k)=F_i(x(0),\dots,x(k))$ where in the design of $F_i$ only limited model information is available about the overall system~(\ref{eqn:sys}). We formalize the notion of what model information is available in the design of controller $i$, $1\leq i\leq N$, through the following definition.

\begin{definition} \label{assmp:3} The design of controller~$i$, $1\leq i\leq N$, has limited model information if (\emph{a})~the exact local realizations $\{A_{ij}(k) \;|\; 1\leq j\leq N,\forall k\}$ are available together with (\emph{b})~the first- and the second-order moments of the system parameters (i.e., $\mathbb{E}\{A(k)\}$ and $\mathbb{E}\{\tilde{A}(k)\otimes \tilde{A}(k)\}$ for all $k$).
\end{definition}

\begin{remark} Note that the assumption that the exact realizations $\{A_{ij}(k) \;|\; 1\leq j\leq N\}$ are available to designer of controller $i$ (and not the rest of the submatrices) is reasonable in the context of interconnected systems where the coupling strengths are known (stochastically-varying or not) and the uncertainties are arising in each subsystem independently. For instance, such systems occur naturally when studying power network control since the power grid, which determines the coupling strengths between the generators and the consumers, is typically accurately modeled, however, the loads and the generators are stochastically varying and uncertain. A direction for future research could be to consider the case where also the coupling strengths are uncertain. 
\end{remark} 

\subsection{Finite-Horizon Cost Function} \label{subsec:FHCF}
In the finite-horizon optimal control design problem, for a fixed $T>0$, we minimize the cost function \begin{equation} \label{eqn:newlabel}
\begin{split}
J_T(x_0,&\{u(k)\}_{k=0}^{T-1})=\mathbb{E}\bigg\{x(T)^\T Q(T)x(T) +\sum_{k=0}^{T-1} \bigg(x(k)^\T Q(k) x(k)+ \sum_{j=1}^N u_j(k)^\T R_{jj}(k)u_j(k) \bigg) \bigg\},
\end{split}
\end{equation}
subject to the system dynamics in~(\ref{eqn:sys}) and the model information constraints in Definition~\ref{assmp:3}. In~(\ref{eqn:newlabel}), we assume that $Q(k)\in\mathcal{S}_{+}^n$ for all $0\leq k\leq T$ and $R(k)=\diag(R_{11}(k),\dots,R_{NN}(k))\in\mathcal{S}_{++}^m$ for all $0\leq k\leq T-1$. The following theorem presents the solution of the finite-horizon optimal control problem.

\begin{theorem} \label{tho:1} The solution of the finite-horizon optimal control design problem with limited model information is given by
\begin{equation} \label{eqn:FHOC}
\begin{split}
u(k)=&-(R(k)+B(k)^\T P(k+1)B(k))^{-1}B(k)^\T P(k+1)\bar{A}(k)x(k)\\ &-\matrix{c}{(R_{11}(k)+B_1(k)^\T P(k+1)B_1(k))^{-1}B_1(k)^\T P(k+1)\tilde{A}_1(k)\\ \vdots \\ (R_{NN}(k)+B_N(k)^\T P(k+1)B_N(k))^{-1}B_N(k)^\T P(k+1)\tilde{A}_N(k)}x(k),
\end{split}
\end{equation}
where the sequence of matrices $\{P(k)\}_{k=0}^T$ can be calculated using the backward difference equation
\begin{equation} \label{eqn:P}
\begin{split}
P(k)=Q(k)+\mathbf{R}(\bar{A}(k),P(k+1),B(k),R)+\sum_{i=1}^N \mathbb{E} \left\{\mathbf{R}(\tilde{A}_i(k),P(k+1),B_i(k),R_{ii}) \right\},
\end{split}
\end{equation}
with the boundary condition $P(T)=Q(T)$. Furthermore,
$\inf_{\{u(k)\}_{k=0}^{T-1}} J_T(x_0,\{u(k)\}_{k=0}^{T-1})=x_0^\T P(0) x_0$.
\end{theorem}

\begin{proof} We solve the finite-horizon optimal control problem using dynamic programming
\begin{equation} \label{eqn:rec}
\begin{split}
V_k(x(k))=\inf_{u(k)} &\mathbb{E}\bigg\{x(k)^\T Q(k)x(k)\hspace{-.04in}+\hspace{-.04in}u(k)^\T R(k)u(k) \hspace{-.04in}+\hspace{-.04in} V_{k+1}(A(k)x(k)+B(k)u(k))\big|x(k) \bigg\},
\end{split}
\end{equation}
where $V_T(x(T))=x(T)^\T Q(T) x(T)$. The proof strategy is to (\textit{a}) show $V_k(x(k))=x(k)^\T P(k)x(k)$ for all $k$ using backward induction, (\textit{b}) find a lower bound for $\mathbb{E}\{x(k)^\T Q(k)x(k)+u(k)^\T R(k)u(k)\linebreak+V_{k+1}(A(k)x(k)+B(k)u(k))\big|x(k)\}$ which is attained by $u(k)$ in~\eqref{eqn:FHOC}, and (\textit{c}) using optimal controller calculate a recursive equation for $P(k)$, $0\leq k\leq T$, starting from $P(T)=Q(T)$. Note that because of Definition~\ref{assmp:3}, in each step of the dynamic programming, the infimum is taken over the set of all control signals $u(k)$ of the form
\begin{equation} \label{control:form}
\matrix{c}{u_1(k) \\ \vdots \\ u_N(k)}=\matrix{c}{\psi_1(A_{11}(k),\dots,A_{1N}(k);x(0),\cdots,x(k)) \\ \vdots \\ \psi_N(A_{N1}(k),\dots,A_{NN}(k);x(0),\cdots,x(k))},
\end{equation}
where $\psi_i:\mathbb{R}^{n_i\times n_1}\times \cdots \times \mathbb{R}^{n_i\times n_N}\times \mathbb{R}^n\rightarrow \mathbb{R}^{m_i}$, $1\leq i\leq N$, can be any mapping (i.e., it is not necessarily a linear mapping, a smooth one, etc). Let us assume, for all $k$, that $V_{k}(x(k))=x(k)^\T P(k) x(k)$
where $P(k)\in\mathcal{S}_+^n$. This is without loss of generality since $V_T(x(T))= x(T)^\T Q(T) x(T)$ is a quadratic function of the state vector $x(T)$ and using dynamic programming, $V_k(x(k))$ remains a quadratic function of $x(k)$ if $V_{k+1}(x(k+1))$ is a quadratic function of $x(k+1)$ and $u(k)$ is a linear function of $x(k)$. This can be easily proved using mathematical induction. For the control input of the form in~\eqref{control:form}, we define 
$$
\bar{G}(k)=\matrix{c}{\mathbb{E}\left\{\psi_1(A_{11}(k),\dots,A_{1N}(k);x(0),\cdots,x(k))|x(k)\right\} \\ \vdots \\ \mathbb{E}\left\{\psi_N(A_{N1}(k),\dots,A_{NN}(k);x(0),\cdots,x(k))|x(k)\right\}}-\bar{K}(k)x(k), 
$$
and 
\begin{equation*}
\begin{split}
\tilde{g}_i(k)=&\psi_i(A_{i1}(k),\dots,A_{iN}(k);x(0),\cdots,x(k)) \\&- \mathbb{E}\{\psi_i(A_{i1}(k),\dots,A_{iN}(k);x(0),\cdots,x(k)) |x(k)\}-\tilde{K}_i(k)x(k), 
\end{split}
\end{equation*}
where $\bar{K}(k)=-(R(k)+B(k)^\T P(k+1)B(k))^{-1} B(k)^\T P(k+1)\bar{A}(k)$ and $\tilde{K}_i(k)=-(R_{ii}(k)+B_i(k)^\T P(k+1)B_i(k))^{-1}B_i(k)^\T P(k+1)\tilde{A}_i(k)$ are the gains in~\eqref{eqn:FHOC}. By definition, we have $\mathbb{E}\{\tilde{g}_i(k)|x(k)\}=0$. Furthermore, let us define the notation 
$$
C_{i}=\left[\hspace{-.06in}\begin{array}{c}0_{(\sum_{j=1}^{i-1} m_j) \times m_i} \\ I \\ 0_{(\sum_{j=i+1}^{N} m_j) \times m_i} \end{array}\hspace{-.06in}\right],
$$
for all $1\leq i\leq N$. Evidently, we have 
\begin{small}
\begin{equation*}
\begin{split}
\matrix{c}{\psi_1(A_{11}(k),\dots,A_{1N}(k);x(0),\cdots,x(k)) \\ \vdots \\ \psi_N(A_{N1}(k),\dots,A_{NN}(k);x(0),\cdots,x(k))}&=
\bar{G}(k)+\matrix{c}{\tilde{g}_1(k) \\ \vdots \\ \tilde{g}_N(k)}+\bar{K}(k)x(k)+\matrix{c}{\tilde{K}_1(k)x(k) \\ \vdots \\ \tilde{K}_N(k)x(k)}\\
&=\bar{G}(k)+\bar{K}(k)x(k)+\sum_{i=1}^N C_i\tilde{g}_i(k)+\sum_{i=1}^N C_i\tilde{K}_i(k)x(k). 
\end{split} 
\end{equation*} 
\end{small}
By rearranging the terms, we can easily show that 
\begin{small}
\begin{equation} \label{inden:u}
\begin{split}
\mathbb{E}\left\{u(k)^\T R(k) u(k)\big|x(k) \right\}=&
\mathbb{E}\bigg\{(\bar{K}(k)x(k)+\bar{G}(k))^\T R(k) (\bar{K}(k)x(k)+\bar{G}(k)) \\ &+ (\bar{K}(k)x(k)+\bar{G}(k))^\T R(k) \left(\sum_{i=1}^N C_i(\tilde{g}_i(k)+\tilde{K}_i(k)x(k))\right)\\ &+ \left(\sum_{i=1}^N C_i(\tilde{g}_i(k)+\tilde{K}_i(k)x(k))\right)^\T R(k) (\bar{K}(k)x(k)+\bar{G}(k)) \\& +\sum_{i=1}^N\sum_{j=1}^N (\tilde{g}_i(k)+\tilde{K}_i(k)x(k))^\T C_i^\T R(k)C_j (\tilde{g}_j(k)+\tilde{K}_j(k)x(k))\big|x(k)\bigg\} \\=& (\bar{K}(k)x(k)+\bar{G}(k))^\T R(k) (\bar{K}(k)x(k)+\bar{G}(k)) \\& + \sum_{i=1}^N \mathbb{E}\bigg\{(\tilde{g}_i(k)+\tilde{K}_i(k)x(k))^\T R_{ii}(k) (\tilde{g}_i(k)+\tilde{K}_i(k)x(k))\big|x(k)\bigg\}, 
\end{split}
\end{equation} 
\end{small}
\hspace{-.1in} where the second equality holds due to that $\mathbb{E}\{\tilde{g}_i(k)+\tilde{K}_i(k)x(k)|x(k)\}=0$ and $C_i^\T RC_j=R_{ij}$ (while recalling that $R_{ij}=0$ if $i\neq j$). Following the same line of reasoning, we show that 
\begin{small}
\begin{equation*}
\begin{split}
\mathbb{E}&\big\{(A(k)x(k)+B(k)u(k))^\T P(k+1) (A(k)x(k)+B(k)u(k))\big|x(k) \big\}\\=&
\mathbb{E}\bigg\{\left(\bar{A}(k)x(k)+B(k)(\bar{K}(k)x(k)+\bar{G}(k))\right)^\T P(k+1) \left(\bar{A}(k)x(k)+B(k)(\bar{K}(k)x(k)+\bar{G}(k))\right) \\ &+ \left(\bar{A}(k)x(k)+B(k)(\bar{K}(k)x(k)+\bar{G}(k))\right)^\T P(k+1) \left(\sum_{i=1}^N \tilde{A}_i(k)x(k)+B_i(k)(\tilde{g}_i(k)+\tilde{K}_i(k)x(k))\right)\\ &+ \left(\sum_{i=1}^N \tilde{A}_i(k)x(k)+B_i(k)(\tilde{g}_i(k)+\tilde{K}_i(k)x(k))\right)^\T P(k+1) \left(\bar{A}(k)x(k)+B(k)(\bar{K}(k)x(k)+\bar{G}(k))\right) \\& +\hspace{-.04in}\sum_{i=1}^N\sum_{j=1}^N \hspace{-.04in}\left( \tilde{A}_i(k)x(k)+B_i(k)(\tilde{g}_i(k)+\tilde{K}_i(k)x(k))\right)^{\hspace{-.04in}\T}\hspace{-.07in} P(k+1)\hspace{-.04in} \left( \tilde{A}_j(k)x(k)+B_j(k)(\tilde{g}_j(k)+\tilde{K}_j(k)x(k))\right)\hspace{-.04in}\big|x(k)\bigg\},
\end{split}
\end{equation*}
\end{small}
where the equality follows from
\begin{equation*}
\begin{split}
A(k)x(k)+B(k)u(k)=&\bar{A}(k)x(k)+B(k)(\bar{K}(k)x(k)+\bar{G}(k))\\&+\sum_{i=1}^N \tilde{A}_i(k)x(k)+B_i(k)(\tilde{g}_i(k)+\tilde{K}_i(k)x(k)).
\end{split}
\end{equation*}
Therefore, we get
\begin{small}
\begin{equation} \label{inden:x}
\begin{split}
\mathbb{E}&\big\{(A(k)x(k)+B(k)u(k))^\T P(k+1) (A(k)x(k)+B(k)u(k))\big|x(k) \big\}\\=& \left(\bar{A}(k)x(k)+B(k)(\bar{K}(k)x(k)+\bar{G}(k))\right)^\T P(k+1) \left(\bar{A}(k)x(k)+B(k)(\bar{K}(k)x(k)+\bar{G}(k))\right) \\&+\hspace{-.04in} \sum_{i=1}^N \hspace{-.04in}\mathbb{E}\bigg\{\hspace{-.04in} \left( \hspace{-.04in} \tilde{A}_i(k)x(k)\hspace{-.03in}+\hspace{-.03in} B_i(k)(\tilde{g}_i(k)\hspace{-.03in}+\hspace{-.03in} \tilde{K}_i(k)x(k))\right)^{\hspace{-.04in}\T} \hspace{-.06in} P(k+1)\hspace{-.04in}\left(\hspace{-.04in} \tilde{A}_i(k)x(k)\hspace{-.03in}+\hspace{-.03in} B_i(k)(\tilde{g}_i(k)\hspace{-.03in}+\hspace{-.03in} \tilde{K}_i(k)x(k))\right)\hspace{-.04in}\big|x(k)\bigg\},
\end{split}
\end{equation}
\end{small}
\hspace{-.15in} because $\tilde{A}_i(k)x(k)+B_i(k)(\tilde{g}_i(k)+\tilde{K}_i(k)x(k))$ and $\tilde{A}_j(k)x(k)+B_j(k)(\tilde{g}_j(k)+\tilde{K}_j(k)x(k))$ are independent random variables for $i\neq j$ (see Assumption~\ref{assmp:iid} and Definition~\ref{assmp:3}) and $\mathbb{E}\{\tilde{A}_i(k)x(k)+B_i(k)(\tilde{g}_i(k)+\tilde{K}_i(k)x(k))|x(k)\}=0$ for all $1\leq i\leq N$. Now, note that
\begin{small}
\begin{equation} \label{ineq:avg}
\begin{split}
\left(\bar{A}(k)x(k)+B(k)(\bar{K}(k)x(k)+\bar{G}(k))\right)^\T &P(k+1) \left(\bar{A}(k)x(k)+B(k)(\bar{K}(k)x(k)+\bar{G}(k))\right)\\&\hspace{.8in}+(\bar{K}(k)x(k)+\bar{G}(k))^\T R(k) (\bar{K}(k)x(k)+\bar{G}(k)) 
\\ &\hspace{-2.2in}= x(k)^\T \bar{K}(k)^\T R(k) \bar{K}(k)x(k) +x(k)^\T (\bar{A}(k)+B(k)\bar{K}(k))^\T P(k+1) (\bar{A}(k)+B(k)\bar{K}(k)) x(k) \\&\hspace{-2.0in} +\bar{G}(k)^\T \left(B(k)^\T P(k+1)(\bar{A}(k)+B(k)\bar{K}(k))+R(k) \bar{K}(k) \right)x(k) \\&\hspace{-2.0in} +x(k)^\T \left(B(k)^\T P(k+1)(\bar{A}(k)+B(k)\bar{K}(k))+R(k) \bar{K}(k) \right)^\T \bar{G}(k) \\&\hspace{-2.0in} +\bar{G}(k)^\T R(k)\bar{G}(k)+\bar{G}(k)^\T B(k)^\T P(k+1) B(k) \bar{G}(k) \\ &\hspace{-2.2in}=x(k)^\T \bar{K}(k)^\T R(k) \bar{K}(k)x(k) +x(k)^\T (\bar{A}(k)+B(k)\bar{K}(k))^\T P(k+1) (\bar{A}(k)+B(k)\bar{K}(k)) x(k) \\&\hspace{-2.0in} +\bar{G}(k)^\T R(k)\bar{G}(k)+\bar{G}(k)^\T B(k)^\T P(k+1) B(k) \bar{G}(k) \\ &\hspace{-2.2in}\geq x(k)^\T \bar{K}(k)^\T R(k) \bar{K}(k)x(k) +x(k)^\T (\bar{A}(k)+B(k)\bar{K}(k))^\T P(k+1) (\bar{A}(k)+B(k)\bar{K}(k)) x(k),
\end{split}
\end{equation}
\end{small}
\hspace{-.06in}where the second equality follows from that $B(k)^\T P(k+1)(\bar{A}(k)+B(k)\bar{K}(k))+R(k) \bar{K}(k)=0$ using the definition of $\bar{K}(k)$ and the inequality holds due to that $\bar{G}(k)^\T (R(k)+ B(k)^\T P(k+1) B(k)) \bar{G}(k)\geq 0$ for any $\bar{G}(k)\in\mathbb{R}^m$ since $R(k)+ B(k)^\T P(k+1) B(k)$ is a positive-definite matrix. Similarly, for each $1\leq i\leq N$, we conclude that
\begin{small}
\begin{equation} \label{ineq:inov}
\begin{split}
&\mathbb{E}\bigg\{(\tilde{g}_i(k)+\tilde{K}_i(k)x(k))^\T R_{ii}(k) (\tilde{g}_i(k)+\tilde{K}_i(k)x(k))\\& +\left( \tilde{A}_i(k)x(k)\hspace{-.03in}+\hspace{-.03in}B_i(k)(\tilde{g}_i(k)\hspace{-.03in}+\hspace{-.03in} \tilde{K}_i(k)x(k))\right)^\T \hspace{-.04in} P(k+1)\left( \tilde{A}_i(k)x(k)\hspace{-.03in}+\hspace{-.03in} B_i(k)(\tilde{g}_i(k)\hspace{-.03in}+\hspace{-.03in} \tilde{K}_i(k)x(k))\right)\big|x(k)\bigg\}
\\
\color{red}
&=\mathbb{E}\bigg\{
\tilde{g}_i(k)^\T R_{ii}(k) \tilde{g}_i(k)+x(k)^\T\tilde{K}_i(k)^\T R_{ii}(k) \tilde{K}_i(k)x(k)+\tilde{g}_i(k)^\T B_i(k)^\T P(k+1)B_i(k)\tilde{g}_i(k)\\
&\hspace{.2in} +\tilde{g}_i(k)^\T \bigg(B_i(k)^\T P(k+1)(\tilde{A}_i(k)+B_i(k)\tilde{K}_i(k))+R_{ii}(k) \tilde{K}_i(k)\bigg)x(k) \\
&\hspace{.2in} +x(k)^\T \bigg(B_i(k)^\T P(k+1)(\tilde{A}_i(k)+B_i(k)\tilde{K}_i(k))+R_{ii}(k) \tilde{K}_i(k)\bigg)^\T\tilde{g}_i(k)\\
&\hspace{.2in} +x(k)^\T (\tilde{A}_i(k)+B_i(k)\tilde{K}_i(k))^\T P(k+1)(\tilde{A}_i(k)+B_i(k)\tilde{K}_i(k))x(k)\big|x(k)\bigg\}
\color{black}
\\& \geq \mathbb{E}\bigg\{x(k)^\T\tilde{K}_i(k)^\T R_{ii}(k) \tilde{K}_i(k)x(k) \\ & \hspace{.9in} +x(k)^\T (\tilde{A}_i(k)+B_i(k)\tilde{K}_i(k))^\T P(k+1)(\tilde{A}_i(k)+B_i(k)\tilde{K}_i(k))x(k)\big|x(k)\bigg\}.
\end{split}
\end{equation}
\end{small}
Combining identities~(\ref{inden:u})--(\ref{inden:x}) with inequalities~(\ref{ineq:avg})--(\ref{ineq:inov})  results in
\begin{small}
\begin{equation*}
\begin{split}
&\mathbb{E}\bigg\{x(k)^\T Q(k)x(k)\hspace{-.04in}+\hspace{-.04in}u(k)^\T R(k)u(k) \hspace{-.04in}+\hspace{-.04in} (A(k)x(k)+B(k)u(k))^\T P(k+1)(A(k)x(k)+B(k)u(k))\big|x(k) \bigg\}\\& \geq x(k)^\T Q(k)x(k)\hspace{-.04in}+\hspace{-.04in}x(k)^\T (\bar{K}(k)^\T R(k) \bar{K}(k)\hspace{-.04in}+\hspace{-.04in}(\bar{A}(k)+B(k)\bar{K}(k))^\T P(k+1) (\bar{A}(k)+B(k)\bar{K}(k))) x(k) \\ &\hspace{.1in} +\sum_{i=1}^N\mathbb{E}\bigg\{\hspace{-.04in}x(k)^\T\hspace{-.04in}(\tilde{K}_i(k)^\T R_{ii}(k) \tilde{K}_i(k)\hspace{-.04in}+\hspace{-.04in}(\tilde{A}_i(k)\hspace{-.04in}+\hspace{-.04in} B_i(k)\tilde{K}_i(k))^\T B_i(k)^\T P(k+1)(\tilde{A}_i(k)\hspace{-.04in}+\hspace{-.04in} B_i(k)\tilde{K}_i(k)))x(k)\big| x(k)\hspace{-.04in}\bigg\}\\&=
\mathbb{E}\bigg\{x(k)^\T Q(k)x(k)\hspace{-.04in}+\hspace{-.04in}u^*(k)^\T R(k)u^*(k) \hspace{-.04in}+\hspace{-.04in} (A(k)x(k)+B(k)u^*(k))^\T P(k+1)(A(k)x(k)\hspace{-.04in}+\hspace{-.04in} B(k)u^*(k))\big|x(k) \bigg\},
\end{split}
\end{equation*}
\end{small}
where
$$
u^*(k)=\bar{K}(k)x(k)+\sum_{i=1}^N C_i \tilde{K}_i(k)x(k).
$$
This inequality proves that $u^*(k)$ is the solution of~(\ref{eqn:rec}) since any other controller results in a larger or equal cost. By substituting this optimal controller inside the recursion~(\ref{eqn:rec}), we get the cost function update equation
\begin{equation} \label{eqn:2}
\begin{split}
x&(k)^\T P(k) x(k)=  \;x(k)^\T Q(k) x(k) \\ &+x(k)^\T\left\{\bar{K}(k)^\T R(k)\bar{K}(k)+(\bar{A}(k)+B(k)\bar{K}(k))^\T P(k+1)(\bar{A}(k)+B(k)\bar{K}(k))\right\}x(k)\\&+\sum_{i=1}^N x(k)^\T\mathbb{E}\hspace{-.04in}\left\{\hspace{-.04in} \tilde{K}_i(k)^\T R_{ii}(k)\tilde{K}_i(k) \hspace{-.04in}+\hspace{-.04in}(\tilde{A}_i(k)\hspace{-.04in}+ \hspace{-.04in}B_i(k)\tilde{K}_i(k))^\T \hspace{-.04in} P(k\hspace{-.04in}+\hspace{-.04in}1)(\tilde{A}_i(k) \hspace{-.04in}+\hspace{-.04in}B_i(k)\tilde{K}_i(k)) \hspace{-.04in} \right\}\hspace{-.04in}x(k),
\end{split}
\end{equation}
By expanding and reordering the terms, we can simplify this equation as
\begin{equation} \label{eqn:2.2}
\begin{split}
x(k)^\T P(k)x(k)
=&\;x(k)^\T Q(k) x(k)+x(k)^\T\mathbf{R}(\bar{A}(k),P(k+1),B(k),R)x(k) \\ &+ \sum_{i=1}^N x(k)^\T\mathbb{E}\left\{\mathbf{R}(\tilde{A}_i(k),P(k+1),B_i(k), R_{ii})\right\}x(k).
\end{split}
\end{equation}
Now, since the equality in~(\ref{eqn:2.2}) is true irrespective of the value of the state vector $x(k)$, we get the recurrence relation in~(\ref{eqn:P}). This concludes the proof.
\end{proof} 

\begin{remark} Theorem~\ref{tho:1} shows that the optimal controller~(\ref{eqn:FHOC}) is a linear state-feedback controller and that it is composed of two parts. The first part is a function of only the parameter statistics (i.e., $\mathbb{E}\{A(k)\}$ and $\mathbb{E}\{\tilde{A}(k)\otimes \tilde{A}(k)\}$) while the second part is a function of exact local model parameters (i.e., $\{A_{ij}(k) \;|\; 1\leq j\leq N\}$ for controller $i$). Note that the optimal controller does not assume any specific probability distribution for the model parameters. It is worth mentioning whenever $n\gg 1$, for computing the optimal controller, we need to perform arithmetic operations on very large matrices (since $\mathbb{E}\{A(k)\}\in\mathbb{R}^{n\times n}$ and $\mathbb{E}\{\tilde{A}(k)\otimes \tilde{A}(k)\}\in\mathbb{R}^{n^2\times n^2}$) which might be numerically difficult (except for special cases where the statistics of the underlying system follows a specific structure or sparsity pattern). 
\end{remark}

\begin{remark} Note that the optimal controller in Theorem~\ref{tho:1} is not structured in terms of the state measurement availability, i.e., controller $i$ accesses the full state measurement $x(k)$. This situation can be motivated for many applications by the rise of fast communication networks that can guarantee the availability of full state measurements in moderately large systems. However, in many scenarios, the model information is simply not available due the fact that each module is being designed separately for commercial purposes without any specific information about its future setup (except the average behavior of other components). A viable direction for future research is to optimize the cost function over the set of structured control laws.
\end{remark}

\begin{remark} It might seem computationally difficult to calculate $\mathbb{E} \{\tilde{A}_i(k)^\T Z \tilde{A}_i(k) \}$ for each time-step $k$ and any given matrix $Z$. However, as pointed out in~\cite{Koning1982}, it suffices to calculate $\mathbb{E} \{\tilde{A}_i(k) \otimes \tilde{A}_i(k) \}$ once, and then use the identity
\begin{equation*}
\begin{split}
\mathbb{E} \{\tilde{A}_i(k)^\T Z(k) \tilde{A}_i(k) \}=\; &
\textrm{vec}^{-1} \left(\mathbb{E} \left\{ \left(\tilde{A}_i(k) \otimes \tilde{A}_i(k)\right)^\T \textrm{vec} \left(Z(k) \right) \right\}\right)\\=\; &
\textrm{vec}^{-1} \left(\mathbb{E} \left\{\tilde{A}_i(k) \otimes \tilde{A}_i(k) \right\}^\T \textrm{vec} \left(Z(k) \right) \right).
\end{split}
\end{equation*}
\end{remark}

\subsection{Infinite-Horizon Cost Function} \label{subsec:IHCF}
In this subsection, we use Theorem~\ref{tho:1} to minimize the infinite-horizon performance criterion
$$
J_\infty(x_0,\{u(k)\}_{k=0}^{\infty})= \lim_{T\rightarrow\infty}J_T(x_0,\{u(k)\}_{k=0}^{T-1}),
$$
where $Q(k)=Q\in\mathcal{S}_{++}^n$ and $R(k)=R\in\mathcal{S}_{++}^m$ for all $0\leq k\leq T-1$ and $Q(T)=0$. For this case, we make the following standing assumption concerning the system parameters statistics:

\begin{assumption} \label{assmp:1} For all time steps $k$, the stochastic processes generating the model parameters of the system in~(\ref{eqn:sys}) satisfy
\begin{itemize}
\item  $\bar{A}(k)=\bar{A}\in\mathbb{R}^{n\times n}$ and $\mathbb{E}\{A(k)\otimes A(k)\}=\Sigma \in\mathbb{R}^{n^2\times n^2}$;
\item $B(k)=B\in\mathbb{R}^{n\times m}$.
\end{itemize}
\end{assumption}

These assumptions are in place to make sure that we are dealing with stationary parameter processes, as otherwise the infinite-horizon optimal control problem could lack physical meaning. We borrow the following technical definition and assumption from~\cite{Koning1982}. We refer interested readers to~\cite{Koning1982} for numerical methods for checking this condition.

\begin{definition} \label{def:1} System~(\ref{eqn:sys}) is called mean square stabilizable if there exists a matrix $L\in\mathbb{R}^{m\times n}$ such that the closed-loop system with controller $u(k)=Lx(k)$ is mean square stable; i.e.,
$\lim_{k\rightarrow +\infty}\mathbb{E}\{x(k)^\T x(k)\}=0.$
\end{definition}

With this definition in hand, we are ready to present the solution of the infinite-horizon optimal control design problem with limited model information.

\begin{theorem} \label{tho:2} Suppose~(\ref{eqn:sys}) satisfies Assumption~\ref{assmp:1} and is mean square stabilizable. The solution of the infinite-horizon optimal control design problem with limited model information is then given by
\begin{equation} \label{eqn:1:tho:2}
\begin{split}
u(k)=&-(R+B^\T PB)^{-1}B^\T P\bar{A}x(k)\\&-\matrix{c}{(R_{11}+B_1^\T PB_1)^{-1}B_1^\T P \tilde{A}_1(k)\\ \vdots \\ (R_{NN}+B_N^\T PB_N)^{-1}B_N^\T P\tilde{A}_N(k)}x(k),
\end{split}
\end{equation}
where $P$ is the unique positive-definite solution of the modified discrete algebraic Riccati equation
\begin{equation} \label{eqn:qRic}
\begin{split}
P=Q+\mathbf{R}(\bar{A},P,B,R)+\sum_{i=1}^N \mathbb{E} \left\{ \mathbf{R}(\tilde{A}_i(k),P,B_i,R_{ii})\right\}.
\end{split}
\end{equation}
Furthermore, the closed-loop system~(\ref{eqn:sys}) and~(\ref{eqn:1:tho:2}) is mean square stable and $$\inf_{\{u(k)\}_{k=0}^{\infty}} J_\infty(x_0,\{u(k)\}_{k=0}^{\infty})=x_0^\T P x_0.$$
\end{theorem}

\begin{proof} Note that the proof of this theorem follows the same line of reasoning as in~\cite{Koning1982}. We extend the result of~\cite{Koning1982} to hold for the Riccati-like backward difference equation presented in~(\ref{eqn:P}). First, let us define the mapping $f:\mathcal{S}_+^n\rightarrow\mathcal{S}_+^n$ such that, for any $X\in\mathcal{S}_+^n$,
\begin{equation*}
\begin{split}
f(X)= Q+ \bar{A}^\T &\left(X - X B (R+B^\T X B)^{-1} B^\T X \right) \bar{A} \\ &+\sum_{i=1}^N \mathbb{E} \left\{ \tilde{A}_i^\T \left(X - X B_i (R_{ii}+B_i^\T X B_i)^{-1} B_i^\T X \right) \tilde{A}_i \right\}.
\end{split}
\end{equation*}
Using part~2 of Subsection~3.5.2 in~\cite{Handbook1996}, we have the matrix inversion identity
$$ X-XW(Z+W^\T XW)^{-1}W^\T X=(X^{-1}+WZ^{-1}W^\T)^{-1}, $$
for any matrix $W$ and positive-definite matrices $X$ and $Z$. Therefore, for any $X\in\mathcal{S}_{++}^n$, we have
\begin{equation} \label{eqn:proof:f_definition}
\begin{split}
f(X)= Q+\bar{A}^\T (X^{-1}+B R^{-1} B^\T)^{-1}\bar{A}+\sum_{i=1}^N \mathbb{E} \left\{ \tilde{A}_i^\T (X^{-1}+B_i R_{ii}^{-1} B_i^\T)^{-1} \tilde{A}_i \right\}.
\end{split}
\end{equation}
Note that, if $X\geq Y\geq 0$, then 
$$
(X^{-1}+W Z^{-1} W^\T)^{-1}\geq (Y^{-1}+W Z^{-1} W^\T)^{-1}, 
$$
for any matrix $W$ and positive-definite matrix $Z$. Therefore, if $X\geq Y\geq 0$, we get 
$$
f(X)\geq f(Y)>0. 
$$
For any given $T\geq 0$, we define the sequence of matrices $\{X_i\}_{i=0}^T$ such that $X_0=0$ and $X_{i+1}=f(X_i)$. We have
$$
X_{1}=f(X_0)=f(0)=Q> 0=X_0. 
$$
Similarly, 
\begin{equation} \label{eqn:proof1}
X_{2}=f(X_{1})\geq f(X_{0})=X_{1} > 0. 
\end{equation}
The left-most inequality in~(\ref{eqn:proof1}) is true because $X_{1}\geq X_{0}$. We can repeat the same argument, and show that for all $1\leq i\leq T-1$, $X_{i+1}\geq X_{i}>0$. Using Theorem~\ref{tho:1}, we know that 
$$
x_0^\T X_T x_0= \inf_{\{u(k)\}_{k=0}^{T-1}} J_T(x_0,\{u(k)\}_{k=0}^{T-1}). 
$$
According to Theorem~5.1 in~\cite{Koning1982} (using the assumption that the underlying system is mean square stabilizable), the sequence $\{X_i\}_{i=0}^\infty$ is uniformly upper-bounded; i.e., there exists $W\in\mathbb{R}^{n\times n}$ such that $X_i\leq W$ for all $i\geq 0$. Therefore, we get
\begin{equation} \label{eqn:proof2}
\lim_{T\rightarrow +\infty} X_T=X\in\mathbb{R}^{n\times n}
\end{equation}
since $\{X_i\}_{i=0}^\infty$ is an increasing and bounded sequence. In addition, we have $X\in\mathcal{S}_{++}^n$ since $X_i\in\mathcal{S}_{++}^n$ for all $i\geq2$ and $\{X_i\}_{i=0}^\infty$ is an increasing sequence. Now, we need to prove that the limit $X$ in~(\ref{eqn:proof2}) is the unique positive definite solution of the modified discrete algebraic Riccati equation~(\ref{eqn:qRic}). This is done by a contrapositive argument. Assume that there exists $Z\in\mathcal{S}_{+}^{n}$ such that $f(Z)=Z$. For this matrix $Z$, we have 
$$
Z=f(Z)\geq f(0)=X_1 
$$
since $Z\geq 0$. Similarly, noting that $Z\geq X_1$, we get 
$$
Z=f(Z)\geq f(X_1)=X_2. 
$$
Repeating the same argument, we get $Z\geq X_i$ for all $i\geq 0$. Therefore, for each $T> 0$, we have the inequality
\begin{equation} \label{eqn:proof3}
\begin{split}
\inf_{\{u(k)\}_{k=0}^{T-1}} J_T(x_0,\{u(k)\}_{k=0}^{T-1})&=x_0^\T X_T x_0 \\ &\leq x_0^\T Z x_0\\ &=\hspace{-.1in} \inf_{\{u(k)\}_{k=0}^{T-1}} \hspace{-.05in} \mathbb{E}\left\{x(T)^\T Z x(T)\hspace{-.03in}+\hspace{-.03in}\sum_{k=0}^{T-1} x(k)^\T Q x(k)\hspace{-.03in} +\hspace{-.03in}u(k)^\T R u(k) \right\}\hspace{-.05in}.
\end{split}
\end{equation}
Note that the last equality in~(\ref{eqn:proof3}) is a direct consequence of Theorem~\ref{tho:1} and the fact that $Z=f^q(Z)$ for any positive $q\in\mathbb{Z}$. Let us define
$\{u^*(k)\}_{k=0}^{T-1}=\arginf_{\{u(k)\}_{k=0}^{T-1}} J_T(x_0,\{u(k)\}_{k=0}^{T-1}),$
and $x^*(k)$ as the state of the system when the control sequence $u^*(k)$ is applied. Now, we get the inequality
\begin{equation} \label{eqn:proof5}
\begin{split}
\inf_{\{u(k)\}_{k=0}^{T-1}} &\mathbb{E}\left\{x(T)^\T Z x(T)+\sum_{k=0}^{T-1} x(k)^\T Q x(k)+u(k)^\T R u(k) \right\}  \\ &\hspace{.7in}\leq \mathbb{E}\left\{x^*(T)^\T Z x^*(T)+\sum_{k=0}^{T-1} x^*(k)^\T Q x^*(k)+u^*(k)^\T R u^*(k) \right\},
\end{split}
\end{equation}
since, by definition, $\{u^*(k)\}_{k=0}^{T-1}$ is not the minimizer of this cost function.
It is easy to see that the right-hand side of~(\ref{eqn:proof5}) is equal to $J_T(x_0,\{u^*(k)\}_{k=0}^{T-1}) + \mathbb{E}\left\{x^*(T)^\T Z x^*(T) \right\}$. Thus, using~(\ref{eqn:proof3}) and~(\ref{eqn:proof5}), we get
\begin{equation} \label{eqn:3}
\begin{split}
x_0^\T X_T x_0 &\leq x_0^\T Z x_0 \\&\leq J_T(x_0,\{u^*(k)\}_{k=0}^{T-1}) + \mathbb{E}\left\{x^*(T)^\T Z x^*(T) \right\}\\&=x_0^\T X_T x_0+\mathbb{E}\left\{x^*(T)^\T Z x^*(T) \right\}.
\end{split}
\end{equation}
Finally, thanks to the facts that $Q>0$ and
\begin{equation*}
\begin{split}
&\lim_{T\rightarrow +\infty}\mathbb{E}\left\{\sum_{k=0}^{T-1} x^*(k)^\T Q x^*(k)+u^*(k)^\T R u^*(k) \right\}=\lim_{T\rightarrow +\infty}x_0^\T X_T x_0 =x_0^\T X x_0 <\infty,
\end{split}
\end{equation*}
we get that
$\lim_{T\rightarrow \infty}\mathbb{E}\left\{x^*(T)^\T x^*(T) \right\}=0.$
Therefore, we have
$\lim_{T\rightarrow \infty}\mathbb{E}\left\{x^*(T)^\T Z x^*(T) \right\}=0.$
Letting $T$ go to infinity in~(\ref{eqn:3}), results in $x_0^\T X x_0=x_0^\T Z x_0$ for all $x_0\in\mathbb{R}^n$. Thus, $X=Z$. This concludes the proof.
\end{proof}

\begin{remark} Note that we can use the procedure introduced in the proof of Theorem~\ref{tho:2} to numerically compute the unique positive-definite solution of the modified discrete algebraic Riccati equation in~(\ref{eqn:qRic}); i.e., we can construct a sequence of matrices $\{X_i\}_{i=0}^\infty$, such that $X_{i+1}=f(X_i)$ with $X_0=0$ where $f(\cdot)$ is defined as in~(\ref{eqn:proof:f_definition}). Because of~(\ref{eqn:proof2}), it is evident that, for each $\delta>0$, there exists a positive integer $q(\delta)$ such that $X_{q(\delta)}$ is in the $\delta$-neighborhood of the unique positive-definite solution of the modified discrete algebraic Riccati equation~(\ref{eqn:qRic}). Hence, the procedure generates a solution with any desired precision. \end{remark}

Note that Definition~\ref{def:1} requires the existence of a fixed feedback gain $L$ that ensures the closed-loop mean square stability. This might result in conservative results. In what follows, we relax this assumption to time-varying matrices.

\begin{definition} \label{def:2} System~(\ref{eqn:sys}) is called mean square stabilizable under limited model information if there exist mappings $\Gamma_i:\mathbb{R}^{n_i\times n_1}\times \cdots \times \mathbb{R}^{n_i\times n_N} \rightarrow \mathbb{R}^{m_i\times n}$, $1\leq i\leq N$, such that the closed-loop system with controller
\begin{equation*}
u(k)=\matrix{c}{\Gamma_1(A_{11}(k),\dots,A_{1N}(k))  \\ \vdots  \\ \Gamma_N(A_{N1}(k),\dots,A_{NN}(k))}x(k),
\end{equation*}
is mean square stable.
\end{definition}

Clearly, if a discrete-time linear system with stochastically-varying parameters is mean square stabilizable, it is also mean square stabilizable under limited model information.

\begin{remark} \label{remark:deadbeat:definition} All fully-actuated systems (i.e., systems where $m_i=n_i$ for all $1\leq i\leq N$) are mean square stabilizable under limited model information because, for each $1\leq i\leq N$, the deadbeat controller $\Gamma_i(A_{i1}(k),\dots,A_{iN}(k))=-B_{ii}^{-1}[A_{i1}(k) \; \cdots \; A_{iN}(k)],$ is based on limited model information and mean square stabilizes the system.
\end{remark}

As a price of relaxing this assumption to time-varying matrices, we need to strengthen Assumption~\ref{assmp:1}.

\begin{assumption} \label{assmp:1a} The stochastic processes generating the model parameters of system~(\ref{eqn:sys}) satisfy that
\begin{itemize}
\item The probability distribution of the matrices $\{A(k)\}_{k=0}^\infty$ is constant in time;
\item $B(k)=B\in\mathbb{R}^{n\times m}$ for all $k\geq 0$.
\end{itemize}
\end{assumption}

Note that in Assumption~\ref{assmp:1} we only needed the first and the second moments of the system parameters to be constant. However, in Assumption~\ref{assmp:1a} all the moments are constant.

\begin{corollary} \label{cor:1} Suppose~(\ref{eqn:sys}) satisfies Assumption~\ref{assmp:1a} and is mean square stabilizable under limited model information. The solution of the infinite-horizon optimal control design problem with limited model information is then given by~(\ref{eqn:1:tho:2}) where $P$ is the unique finite positive-definite solution of the modified discrete algebraic Riccati equation in~(\ref{eqn:qRic}). Furthermore, the closed-loop system~(\ref{eqn:sys}) and~(\ref{eqn:1:tho:2}) is mean square stable and $\inf_{\{u(k)\}_{k=0}^{\infty}} J_\infty(x_0,\{u(k)\}_{k=0}^{\infty})=x_0^\T P x_0.$
\end{corollary}

\begin{proof} The only place in the proof of Theorem~\ref{tho:2} where we used the assumption that the underlying system is mean square stabilizable, was to show that the sequence $\{X_i\}_{i=0}^\infty$ is upper bounded; i.e., there exists $W\in\mathcal{S}_{+}^{n}$ such that $X_i\leq W$ for all $i\geq 0$. We just need to prove this fact considering the assumption that the system is mean square stabilizable under limited model information. Note that for any $T> 0$, we have
\begin{equation} \label{eqn:0:proof:cor:1}
\begin{split}
\inf_{\{u(k)\}_{k=0}^{T-1}} J_T(x_0,\{u(k)\}_{k=0}^{T-1})&=x_0^\T X_T x_0 \leq \mathbb{E}\left\{\sum_{k=0}^{T-1} x(k)^\T Q x(k)+\bar{u}(k)^\T R \bar{u}(k) \right\},
\end{split}
\end{equation}
where $x(k)$ is the system state when it is initialized at $x(0)=x_0$ and the control law $\bar{u}(k)=\Gamma(k)x(k)$ is in effect with
$$
\Gamma(k)=\matrix{c}{\Gamma_1(A_{11}(k),\dots,A_{1N}(k)) \\ \vdots \\ \Gamma_N(A_{N1}(k),\dots,A_{NN}(k))}
$$
that satisfies the condition of Definition~\ref{def:2}. Note that, at each time step $k$, $\Gamma(k)$ is independent of $x(k)$ because of Assumption~\ref{assmp:iid}. Therefore, we have
\begin{equation*}
\begin{split}
\mathbb{E}\left\{\sum_{k=0}^{T-1} x(k)^\T Q x(k)+\bar{u}(k)^\T R \bar{u}(k) \right\}&=
\mathbb{E}\left\{\sum_{k=0}^{T-1} x(k)^\T (Q+\Gamma(k)^\top R\Gamma(k)) x(k) \right\}\\&=
\mathbb{E}\left\{\sum_{k=0}^{T-1} x(k)^\T (Q+\mathbb{E}\{\Gamma(k)^\top R\Gamma(k)\}) x(k) \right\}.
\end{split}
\end{equation*}
Furthermore, we can see that $\mathbb{E}\{\Gamma(k)^\top R\Gamma(k)\}=\bar{R}\in\mathcal{S}_{+}^n$ due to Assumption~\ref{assmp:1a}. Now, let us define the sequence $\{W_i\}_{i=0}^\infty$ such that $W_0=Q+\bar{R}$ and $W_{i+1}=\mathbb{E}\{(A(i)+B\Gamma(i))^\top W_i(A(i)+B\Gamma(i))\}$ which results in 
\begin{equation*}
\begin{split}
\mathbb{E}\left\{\sum_{k=0}^{T-1} x(k)^\T Q x(k)+\bar{u}(k)^\T R \bar{u}(k) \right\}&=
\mathbb{E}\left\{\sum_{k=0}^{T-1} x_0^\T W_k x_0 \right\}=x_0^\T\mathbb{E}\left\{\sum_{k=0}^{T-1}  W_k  \right\}x_0.
\end{split}
\end{equation*}
Notice that by construction, $W_i\geq 0$ for all $i$. In what follows, we prove that $\lim_{T\rightarrow \infty}\sum_{k=0}^{T-1} W_k=W<\infty$. Notice that using Assumption~\ref{assmp:1a}, we have $\mathbb{E}\{(A(i)+B\Gamma(i))^\top\otimes(A(i)+B\Gamma(i))^\top\}=\bar{U}$ for a fixed matrix $\bar{U}\in\mathbb{R}^{n^2\times n^2}$. 

\textsc{Claim 1}: $\max_j|\lambda_j(\bar{U})|<1$ where $\lambda_j(\cdot)$ denotes the eigenvalues of a matrix.

To prove this claim, construct a sequence $\{\bar{W}_i\}_{i=0}^\infty$ such that $\bar{W}_{i+1}=\mathbb{E}\{(A(i)+B\Gamma(i))^\top \bar{W}_i\linebreak(A(i)+B\Gamma(i))\}$ and $\bar{W}_0$ can be an arbitrary matrix (note that the difference between $\{W_i\}_{i=0}^\infty$ and $\{\bar{W}_i\}_{i=0}^\infty$ is the initial condition). Now, using an inductive argument, we prove that $\bar{W}_k=\vect^{-1}(\bar{U}^k\vect(\bar{W}_0))$. Firstly, 
\begin{align*}
\bar{W}_1
&=\mathbb{E}\left\{(A(1)+B\Gamma(1))^\top \bar{W}_0 (A(1)+B\Gamma(1))\right\}\\
&=\mathbb{E}\bigg\{\vect^{-1}\bigg((A(1)+B\Gamma(1))^\top\otimes(A(1)+B\Gamma(1))^\top \vect(\bar{W}_0)\bigg)\bigg\}\nonumber\\
&=\vect^{-1}\bigg(\mathbb{E}\bigg\{(A(1)+B\Gamma(1))^\top\otimes(A(1)+B\Gamma(1))^\top\bigg\} \vect(\bar{W}_0)\bigg)\bigg\}\nonumber\\
&=\vect^{-1}(\bar{U} \vect(\bar{W}_0)).
\end{align*}
where the second equality follows from the fact that for any three compatible matrices $A,B,C$, we have $ABC=\vect^{-1}((C^\T \otimes A) \vect(B))$ and the third equality holds because $\vect^{-1}$ is a linear operator.
Now, let us show that $\bar{W}_{k+1}=\vect^{-1}(\bar{U}^{k+1} \vect(\bar{W}_0))$ if $\bar{W}_k=\vect^{-1}(\bar{U}^k\vect(\bar{W}_0))$. To do so, notice that
\begin{align*}
\bar{W}_{k+1}
&=\mathbb{E}\left\{(A(k+1)+B\Gamma(k+1))^\top \bar{W}_k (A(k+1)+B\Gamma(k+1))\right\}\\
&=\mathbb{E}\bigg\{\vect^{-1}\bigg((A(k+1)+B\Gamma(k+1))^\top\otimes(A(k+1)+B\Gamma(k+1))^\top \vect(\bar{W}_k)\bigg)\bigg\}\nonumber\\
&=\vect^{-1}\bigg(\mathbb{E}\bigg\{(A(k+1)+B\Gamma(k+1))^\top\otimes(A(k+1)+B\Gamma(k+1))^\top\bigg\} \bar{U}^k\vect(\bar{W}_0)\bigg)\bigg\}\nonumber\\
&=\vect^{-1}(\bar{U}^{k+1} \vect(\bar{W}_0)).
\end{align*}
This conclude the induction. Now, notice that $\lim_{k\rightarrow \infty} x_0^\T \bar{W}_k x_0=\lim_{k\rightarrow \infty}\mathbb{E}\{x(k)^\T \bar{W}_0 x(k)\}=0$ for any $x_0\in\mathbb{R}^n$ because $\Gamma(k)$ satisfies the condition of Definition~\ref{def:2}. As a result, $\lim_{k\rightarrow \infty} \bar{W}_k =0$. Therefore, we get $\lim_{k\rightarrow \infty} \bar{U}^k\vect(\bar{W}_0)=0$ irrespective of the choice of $\bar{W}_0$ which, in turn, implies that $\lim_{k\rightarrow \infty} \bar{U}^k=0$. Using Theorem~4~\cite[p.\,14]{isaacson1994analysis}, we get $\max_j|\lambda_j(\bar{U})|<1$.

Now that we have proved Claim~1, we are ready to show that $\lim_{T\rightarrow \infty}\sum_{k=0}^{T-1} W_k=W<\infty$. Recalling the proof of Claim~1 while setting $\bar{W}_0=W_0=Q+\bar{R}$, we get that $W_k=\vect^{-1}(\bar{U}^k\vect(Q+\bar{R}))$ and as a result
\begin{align*}
\lim_{T\rightarrow \infty}\sum_{k=0}^{T-1} W_k
&=\lim_{T\rightarrow \infty}\sum_{k=0}^{T-1} \vect^{-1}(\bar{U}^k\vect(Q+\bar{R}))
=\vect^{-1}\left(\left[\lim_{T\rightarrow \infty}\sum_{k=0}^{T-1} \bar{U}^k\right]\vect(Q+\bar{R})\right).
\end{align*}
Now, notice that using Claim~1, $\lim_{T\rightarrow \infty}\sum_{k=0}^{T-1} \bar{U}^k=(I-\bar{U})^{-1}$. Let $\bar{U}^\infty=(I-\bar{U})^{-1}\in\mathbb{R}^{n^2\times n^2}$. Hence, we get $\lim_{T\rightarrow \infty}\sum_{k=0}^{T-1} W_k=\vect^{-1}(\bar{U}^\infty\vect(Q+\bar{R}))<\infty$. Let us define $W=\vect^{-1}(\bar{U}^\infty\linebreak\vect(Q+\bar{R}))$. Using~(\ref{eqn:0:proof:cor:1}), we get
\begin{align*}
x_0^\T X_T x_0 \leq \mathbb{E}\left\{\sum_{k=0}^{T-1} x(k)^\T Q x(k)+\bar{u}(k)^\T R \bar{u}(k) \right\}
=\sum_{k=0}^{T-1} x_0^\T W_k x_0
\leq  \sum_{k=0}^{\infty} x_0^\T W_k x_0 \leq x_0^\top W x_0.
\end{align*}
This inequality is indeed true irrespective of the initial condition $x_0$ and the time horizon $T$. Therefore, $X_i\leq W$ for all $i\geq 0$. The rest of the proof is similar to that of Theorem~\ref{tho:2}.
\end{proof}

\setcounter{example}{0}
\begin{example}[Cont'd] \label{example:2} Let us introduce the quadratic cost function
$$
J_\infty(x_0,\{u(k)\}_{k=0}^{\infty})=\mathbb{E}\left\{\sum_{k=0}^\infty x(k)^\top x(k)+u(k)^\top u(k) \right\}.
$$
Following Theorem~\ref{tho:2}, we can easily calculate the optimal controller with limited model information as
\begin{equation*}
\begin{split}
u^{\scriptsize{\mbox{LMI}}}(k)=\matrix{cccc}{
   42.7701+8.0694\alpha_1(k)  & -1.6741 & -29.1868 &   0.1041 \\
  -23.2274  &  0.1757 &  34.4246+6.8698\alpha_2(k) &  -1.7331}x(k).
\end{split}
\end{equation*}
Clearly, the control gain $L_i\in\mathbb{R}^{1\times 4}$ of controller $u_i(k)=L_i(k)x(k)$, $i=1,2$, is a function of only its corresponding subsystem's model parameter $\alpha_i(k)$. \hfill\newqed
\end{example}

An interesting question is what is the value of model information when designing an optimal controller; i.e., having only access to local model information how much does the closed-loop performance degrade in comparison to having access to global model information. To answer this question for the setting considered in this paper, we need to introduce the optimal control design with full model information.

\section{Control Design with Full Model Information} \label{sec:fullmodelinfo}
In this section, we consider the case where we have access to the full model information when designing each subcontroller. Hence, we make the following definition:
\begin{definition} \label{assmp:3:complete} The design of controller~$i$, $1\leq i\leq N$, has full model information if (\emph{a})~the entire model parameters $\{A_{ij}(k) \;|\; 1\leq i,j\leq N,\forall k\}$ are available together with (\emph{b})~the first- and the second-order moments of the system parameters (i.e., $\mathbb{E}\{A(k)\}$ and $\mathbb{E}\{\tilde{A}(k)\otimes \tilde{A}(k)\}$ for all $k$).
\end{definition}

We have the following result for the finite-horizon case.

\begin{theorem} \label{tho:3} The solution of the finite-horizon optimal control design problem with full model information is given by
\begin{equation}
\begin{split}
u(k)=-(R+B(k)^\T &P(k+1)B(k))^{-1}B(k)^\T P(k+1)A(k)x(k),
\end{split}
\end{equation}
where $\{P(k)\}_{k=0}^T$ can be found using the backward difference equation
\begin{equation} \label{eqn:recursive:tho:3}
\begin{split}
P(k)=Q(k)+\mathbf{R}(\bar{A}(k),P(k+1),B(k),R)+\sum_{i=1}^N \mathbb{E} \left\{\mathbf{R}(\tilde{A}_i(k),P(k+1),B(k),R) \right\},
\end{split}
\end{equation}
with the boundary condition $P(T)=Q(T)$. Furthermore, $\inf_{\{u(k)\}_{k=0}^{T-1}} J_T(x_0,\{u(k)\}_{k=0}^{T-1})=x_0^\T P(0) x_0$.
\end{theorem}

\begin{proof} 
\ifdefined\SHORTVERSION
The proof is similar to the proof of Theorem~\ref{tho:1} and is therefore omitted. See~\cite{FarokhiCompleteManuscript2013} for the detailed proof.
\fi
\ifdefined\LONGVERSION
The proof is similar to the proof of Theorem~\ref{tho:1}. We solve the finite-horizon optimal control problem using dynamic programming
\begin{equation}  \label{eqn:dynamicprogramming:tho:3}
\begin{split}
V_k(x(k))=\inf_{u(k)} &\mathbb{E}\bigg\{x(k)^\T Q(k)x(k)\hspace{-.04in}+\hspace{-.04in}u(k)^\T R(k)u(k) \hspace{-.04in}+\hspace{-.04in} V_{k+1}(A(k)x(k)+B(k)u(k))\big|x(k) \bigg\},
\end{split}
\end{equation}
where $V_T(x(T))=x(T)^\T Q(T) x(T)$. Note that because of Definition~\ref{assmp:3}, in each step of the dynamic programming, the infimum is taken over the set of all control signals $u(k)$ of the form
\begin{equation*}
\matrix{c}{u_1(k) \\ \vdots \\ u_N(k)}=\matrix{c}{\psi_1(A(k);x(0),\dots,x(k)) \\ \vdots \\ \psi_N(A(k);x(0),\dots,x(k))},
\end{equation*}
where $\psi_i:\mathbb{R}^{n\times n}\times \mathbb{R}^n \times \cdots \times \mathbb{R}^n\rightarrow \mathbb{R}^{m_i}$, $1\leq i\leq N$, can be any mapping. Notice the difference between this proof and that of Theorem~\ref{tho:1} is the fact that $\psi_i$ are a function of the entire matrix $A(k)$. Similarly, let us assume, for all $k$, that $V_{k}(x(k))=x(k)^\T P(k) x(k)$ where $P(k)\in\mathcal{S}_+^n$. This is without loss of generality since $V_T(x(T))= x(T)^\T Q(T) x(T)$ is a quadratic function of the state vector $x(T)$ and using dynamic programming, $V_k(x(k))$ remains a quadratic function of $x(k)$ if $V_{k+1}(x(k+1))$ is a quadratic function of $x(k+1)$ and $u(k)$ is a linear function of $x(k)$. Let us define $K(k)=-(R+B(k)^\T P(k+1)B(k))^{-1}B(k)^\T P(k+1)A(k)$ and $\tilde{u}(k)=u(k)-K(k)x(k)$. This can be easily proved using mathematical induction. Now, we have
\begin{align*}
\mathbb{E} \bigg\{ x(k)^\T Q(k)&x(k)+u(k)^\T R(k)u(k) + (A(k)x(k)+B(k)u(k))^\T P(k+1)(A(k)x(k)+B(k)u(k)) \big| x(k)\bigg\}
\\
=\mathbb{E} \bigg\{ &x(k)^\T Q(k)x(k)+(\tilde{u}(k)+K(k)x(k))^\T R(k)(\tilde{u}(k)+K(k)x(k)) \\&+ (A(k)x(k)+B(k)(\tilde{u}(k)+K(k)x(k)))^\T P(k+1) \\ & \,\,\,\,\times (A(k)x(k)+B(k)(\tilde{u}(k)+K(k)x(k))) \big| x(k)\bigg\} \\
=\mathbb{E} \bigg\{ &x(k)^\T Q(k)x(k)+\tilde{u}(k)^\T R(k)\tilde{u}(k)+x(k)^\T K(k)^\T R(k) K(k) x(k)\\&+x(k)^\T(A(k)+B(k)K(k))^\T P(k+1)(A(k)+B(k)K(k))x(k)
\\ &+\tilde{u}(k)^\T [B(k)^\T P(k+1)(A(k)+B(k)K(k))+R(k)K(k)]x(k)
\\ &+x(k)^\T [B(k)^\T P(k+1)(A(k)+B(k)K(k))+R(k)K(k)]^\T\tilde{u}(k)
\\ &+\tilde{u}(k)^\T B(k)^\T P(k+1)B(k)\tilde{u}(k) \big| x(k)\bigg\} 
\\ \geq \mathbb{E} \bigg\{ &x(k)^\T Q(k)x(k)+ x(k)^\T K(k)^\T R(k) K(k) x(k)\\&+x(k)^\T(A(k)+B(k)K(k))^\T P(k+1)(A(k)+B(k)K(k))x(k)\big| x(k)\bigg\} 
\end{align*}
where the inequality follows from the facts that $B(k)^\T P(k+1)(A(k)+B(k)K(k))+R(k)K(k))=0$, by definition of $K(k)$, and that $\tilde{u}(k)^\T (R(k)+B(k)^\T P(k+1)B(k))\tilde{u}(k)\geq 0$. This indeed proves that $u(k)=K(k)x(k)$ is a minimizer of $\mathbb{E}\{x(k)^\T Q(k)x(k)\hspace{-.04in}+\hspace{-.04in}u(k)^\T R(k)u(k) \hspace{-.04in}+\hspace{-.04in} V_{k+1}(A(k)x(k)+B(k)u(k))|x(k)\}$. Now, the recursive update equation in~\eqref{eqn:recursive:tho:3} can be readily extracted from plugging in the optimal controller $u(k)=K(k)x(k)$ into~\eqref{eqn:dynamicprogramming:tho:3}. This concludes the proof.
\fi
\end{proof}

This result can be extended to the infinite-horizon cost function. However, we first need to present the following definition.

\begin{definition} \label{def:3} System~(\ref{eqn:sys}) is called mean square stabilizable under full model information if there exists a mapping $\Gamma:\mathbb{R}^{n\times n} \rightarrow \mathbb{R}^{m\times n}$ such that the closed-loop system with controller $u(k)=\Gamma(A(k))x(k)$ is mean square stable.
\end{definition}

\begin{theorem} \label{tho:4} Suppose~(\ref{eqn:sys}) satisfies Assumption~\ref{assmp:1} and is mean square stabilizable under full model information. The solution of the infinite-horizon optimal control design problem with full model information is then given by
\begin{equation}
\begin{split}
u(k)=-(R+B^\T PB)^{-1}B^\T PA(k)x(k),
\end{split}
\end{equation}
where $P$ is the unique finite positive-definite solution of the modified discrete algebraic Riccati equation
\begin{equation} \label{eqn:modRiccati}
\begin{split}
P=Q+\mathbf{R}(\bar{A},P,B,R)+\sum_{i=1}^N \mathbb{E} \left\{ \mathbf{R}(\tilde{A}_i(k),P,B,R)\right\}.
\end{split}
\end{equation}
Furthermore, this controller mean square stabilizes the system and
$\inf_{\{u(k)\}_{k=0}^{\infty}} J_\infty(x_0,\{u(k)\}_{k=0}^{\infty})=x_0^\T P x_0.$
\end{theorem}

\begin{proof} The proof is similar to the proofs of Theorem~\ref{tho:2} and Corollary~\ref{cor:1}. \end{proof}

\setcounter{example}{0}
\begin{example}[Cont'd] \label{example:3}  Following Theorem~\ref{tho:4}, the optimal control design with full model information is
\begin{equation*}
\begin{split}
u^{\scriptsize{\mbox{FMI}}}(k)=\matrix{cccc}{
   42.7701+7.9708\alpha_1(k) &  -1.6741 & -29.1868-0.1035\alpha_2(k) &   0.1041 \\
  -23.2274-0.1215\alpha_1(k) &   0.1757 &  34.4246+6.7725\alpha_2(k) &  -1.7330}x(k).
\end{split}
\end{equation*}
Note that the gain of controller~$i$ depends on the global model parameters. \hfill\newqed 
\end{example}

\section{Performance Degradation under Model Information Limitation} \label{sec:Competitive Ratio}
In this section, we study the value of the plant model information using the closed-loop performance degradation caused by lack of full model information in the control design procedure. The performance degradation is captured using the ratio of the closed-loop performance of the optimal controller with limited model information to the closed-loop performance of the optimal controller with global plant model information. Let $\{u^{\scriptsize{\mbox{LMI}}}(k)\}_{k=0}^\infty$ and $\{u^{\scriptsize{\mbox{FMI}}}(k)\}_{k=0}^\infty$ denote the optimal controller with limited model information (Theorem~\ref{tho:2}) and the optimal controller with full model information (Theorem~\ref{tho:4}), respectively. We define the performance degradation ratio as
$$
r=\sup_{x_0\in\mathbb{R}^n}\frac{J_\infty(x_0,\{u^{\scriptsize{\mbox{LMI}}}(k)\}_{k=0}^{\infty})}
{J_\infty(x_0,\{u^{\scriptsize{\mbox{FMI}}}(k)\}_{k=0}^{\infty})}.
$$
Note that $r\geq 1$ since the optimal controller with full model information always outperforms the optimal controller with limited model information.
\setcounter{example}{0}
\begin{example}[Cont'd] \label{example:4} In this example, we compare the closed-loop performance of the optimal controllers under different information regimes. We have already calculated the optimal controller with limited model information as well as the optimal controller with full model information for this numerical example. Now, let us find the optimal controller using statistical model information based on~\cite{Koning1982}. Using Theorem~5.2 from~\cite{Koning1982}, we get
\begin{equation*}
\begin{split}
u^{\scriptsize{\mbox{SMI}}}(k)=\matrix{cccc}{
  41.9043 &  -1.7873 & -29.3969 &  -0.0121 \\
  -23.3180 &  0.0435 & 32.7901 &  -1.8779}x(k).
\end{split}
\end{equation*}
Note how these three control laws depend on the plant model parameters. The control $u^{\scriptsize{\mbox{SMI}}}(k)$ has a static gain depending on the statistical information of the $A$-matrix, while $u^{\scriptsize{\mbox{FMI}}}(k)$ and $u^{\scriptsize{\mbox{LMI}}}(k)$ depend on the actual realizations of the stochastic parameters. Now, we can explicitly compute the performance degradation ratio
$$
r=\sup_{x_0\in\mathbb{R}^n}\frac{x_0^\T P^{\scriptsize{\mbox{LMI}}} x_0}
{x_0^\T P^{\scriptsize{\mbox{FMI}}} x_0}=1+2.266\times 10^{-4}.
$$
This shows that the performance of the optimal controller with limited model information is practically the same as the performance of the optimal controller with full model information. It is interesting to note that with access to (precise) local model information, one can expect a huge improvement in the closed-loop performance in comparison to the optimal controller with only statistical model information because
$$
\sup_{x_0\in\mathbb{R}^n}\frac{x_0^\T P^{\scriptsize{\mbox{SMI}}} x_0}
{x_0^\T P^{\scriptsize{\mbox{LMI}}} x_0}=5.8790.
$$
\hfill\newqed
\end{example}

Next we derive an upper bound for the performance degradation ratio $r$. We do that for fully-actuated systems.

\begin{assumption} \label{assmp:2} All subsystems~(\ref{eqn:eachsubsystem}) are fully-actuated; i.e., $B_{ii}\in\mathbb{R}^{n_i\times n_i}$ and $\underline{\sigma}(B_{ii})\geq \epsilon>0$ for all $1\leq i\leq N$, where $\underline{\sigma}(\cdot)$ denotes the smallest singular value of a matrix.
\end{assumption}

To simplify the presentation, we also assume that $Q=R=I$. This is without loss of generality since the change of variables $(x',u')= (Q^{1/2}x,R^{1/2}u)$ transforms the cost function and state space representation into
\begin{equation*}
J_\infty(x_0,\{u'(k)\}_{k=0}^{\infty})=\lim_{T\rightarrow \infty} \mathbb{E}\left\{\sum_{k=0}^{T-1} x'(k)^\top x'(k)+u'(k)^\top u'(k)\right\}\hspace{-.04in},
\end{equation*}
and
\begin{equation*}
\begin{split}
x'(k+1)&=Q^{1/2}A(k)Q^{-1/2}x'(k)+Q^{1/2}BR^{-1/2}u'(k)=A'(k)x'(k)+B'u'(k).
\end{split}
\end{equation*}

The next theorem presents an upper bound for the performance degradation.

\begin{theorem} \label{tho:upper} Suppose~(\ref{eqn:sys}) satisfies Assumptions~\ref{assmp:1} and~\ref{assmp:2} and is mean square stabilizable under limited model information. The performance degradation ratio is then upper bounded as $r\leq 1+1/\epsilon^2$ where $\epsilon>0$ is defined in Assumption~\ref{assmp:2}.
\end{theorem}

\begin{proof} Using the modified discrete algebraic Riccati equation~(\ref{eqn:modRiccati}) in Theorem~\ref{tho:4}, the cost of the optimal control design with full model information $J_\infty(x_0,\{u^{\scriptsize{\mbox{FMI}}}(k)\}_{k=0}^{\infty})=x_0^\T P^{\scriptsize{\mbox{FMI}}} x_0$ is equal to
\begin{equation} \label{eqn:proofcost}
\begin{split}
x_0^\T P^{\scriptsize{\mbox{FMI}}}& x_0=\;x_0^\T Qx_0+x_0^\T \mathbf{R}(\bar{A},P^{\scriptsize{\mbox{FMI}}},B,I)x_0+\sum_{i=1}^N x_0^\T \mathbb{E}  \left\{ \mathbf{R}(\tilde{A}_i(k),P^{\scriptsize{\mbox{FMI}}},B,I)\right\} x_0.
\end{split}
\end{equation}
In addition, we know that $P^{\scriptsize{\mbox{FMI}}}\geq Q=I$, which (using the proof of Theorem~\ref{tho:2}) results in
\begin{eqnarray}
\mathbf{R}(\bar{A},P^{\scriptsize{\mbox{FMI}}},B,I)\hspace{-.08in}&\geq&\hspace{-.08in} \mathbf{R}(\bar{A},I,B,I), \label{eqn:proof4.1} \\
\mathbf{R}(\tilde{A}_i(k),P^{\scriptsize{\mbox{FMI}}},B,I) \hspace{-.08in}&\geq&\hspace{-.08in} \mathbf{R}(\tilde{A}_i(k),I,B,I). \label{eqn:proof4.2}
\end{eqnarray}
Substituting~(\ref{eqn:proof4.1})--(\ref{eqn:proof4.2}) inside~(\ref{eqn:proofcost}) gives
\begin{equation*}
\begin{split}
x_0^\T P^{\scriptsize{\mbox{FMI}}} x_0
&\geq x_0^\T(I+\bar{A}^\T (I+BB^\T)^{-1} \bar{A})x_0+\sum_{i=1}^N x_0^\T \mathbb{E} \left\{ \tilde{A}_i(k)^\T (I+BB^\T)^{-1} \tilde{A}_i(k)\right\}x_0\\ 
&=    x_0^\T x_0 + x_0^\T\mathbb{E}\{A(k)^\T(I+BB^\T)^{-1}A(k)\}x_0, 
\end{split}
\end{equation*}
where the equality follows from the fact that $\tilde{A}_i(k)$ and $\tilde{A}_j(k)$ for $i\neq j$ are independent random variables with zero mean. On the other hand, for a given $x_0\in\mathbb{R}^n$, the cost of the optimal control design with limited model information $J_\infty(x_0,\{u^{\scriptsize{\mbox{LMI}}}(k)\}_{k=0}^{\infty})=x_0^\T P^{\scriptsize{\mbox{LMI}}} x_0$ is upper-bounded by
$$
x_0^\T P^{\scriptsize{\mbox{LMI}}} x_0 \leq \mathbb{E}\left\{\sum_{k=0}^{+\infty} x(k)^\T x(k)+u(k)^\T u(k) \right\},
$$
where $u(k)=-B^{-1}A(k)x(k)$ and $x(k)$ is the state vector of the system when this control sequence is applied to the system. This is true since the deadbeat control design strategy $u(k)=-B^{-1}A(k)x(k)$ uses only local model information for designing each controller~\cite{Farokhi-thesis2012}. Therefore, 
\begin{equation*}
\begin{split}
x_0^\T P^{\scriptsize{\mbox{LMI}}} x_0 &\leq \mathbb{E}\left\{ x_0^\T (I+A(k)^\T B^{-\T}B^{-1}A(k)) x_0 \right\}.
\end{split}
\end{equation*}
Let us define the set $\mathcal{M}_r=\left\{ \bar{\beta}\in\mathbb{R} \;|\; r\leq \bar{\beta} \right\}$ where $r$ is the performance degradation ratio. If $\beta\in\mathbb{R}$ satisfy $\beta P^{\scriptsize{\mbox{FMI}}}-P^{\scriptsize{\mbox{LMI}}}\geq 0$, then $\beta\in\mathcal{M}_r$. We have
\begin{equation} \label{eqn:PFMIPLMI}
\beta P^{\scriptsize{\mbox{FMI}}}-P^{\scriptsize{\mbox{LMI}}} \geq (\beta-1) I+
\mathbb{E}\{A(k)^\T\left[\beta(I+BB^\T)^{-1}-B^{-\T}B^{-1}\right]A(k)\}.
\end{equation}
Note that if $\beta\geq 1+1/\epsilon^2$, we get $\beta(I+B_{ii}B_{ii}^\T)^{-1} -B_{ii}^{-\T}B_{ii}^{-1}\geq 0$ and therefore, $\beta(I+BB^\T)^{-1}-B^{-\T}B^{-1}\geq 0$. As a result, if $\beta\geq 1+1/\epsilon^2$, the right hand side of~\eqref{eqn:PFMIPLMI} is a positive-semidefinite matrix and, subsequently, $\beta P^{\scriptsize{\mbox{FMI}}}-P^{\scriptsize{\mbox{LMI}}} \geq 0$. Hence, $[1+1/\epsilon^2,+\infty)\subseteq \mathcal{M}_r$. This shows that
$$
r=\sup_{x_0\in\mathbb{R}^n}\frac{x_0^\T P^{\scriptsize{\mbox{LMI}}} x_0}{x_0^\T P^{\scriptsize{\mbox{FMI}}} x_0}\leq 1+\frac{1}{\epsilon^2}.
$$
\end{proof}

As the power network in Example~\ref{example:1} is not fully-actuated, we consider another power network example to the illustrate the previous result.

\begin{example} Consider DC power generators, such as solar farms and batteries. Suppose these sources are connected to AC transmission lines through DC/AC converters that are equipped with a droop-controller~\cite{fd-fb:12i,simpson2012droop}. Let us assume that both power generators in Figure~1 are such DC power generators equipped with droop-controlled converters. We can then model this power network as 
\begin{equation*}
\begin{split}
\dot{\delta}_1(t)&=\frac{1}{D_1}\big[P_{1}(t)-c_{12}^{-1}\sin(\delta_1(t)-\delta_2(t)) -c_{1}^{-1}\sin(\delta_1(t))-D_1\omega_1(t) \big], \\
\dot{\delta}_2(t)&=\frac{1}{D_2}\big[P_{2}(t)-c_{12}^{-1}\sin(\delta_2(t)-\delta_1(t)) -c_{2}^{-1}\sin(\delta_2(t))-D_2\omega_2(t) \big],
\end{split}
\end{equation*}
where $\delta_i(t)$, $1/D_i>0$, and $P_{i}(t)$ are respectively the phase angle of the terminal voltage of converter~$i$, its converter droop-slope, and its input power. The power network parameters in this example are the same as the ones in Example~\ref{example:1}, except $D_1=D_2=1.0$. Now, similarly to Example~\ref{example:1}, we find the equilibrium point of this nonlinear system, linearize it around this equilibrium, and then, discretize the system with sampling time $\Delta T=300\,\mathrm{ms}$ to get  
\begin{equation*}
\begin{split}
\hspace{-.04in}\matrix{c}{ \hspace{-.08in} \Delta\delta_1(k+1)  \hspace{-.08in} \\ \hspace{-.08in} \Delta\delta_2(k+1) \hspace{-.08in}} \hspace{-.06in}=\hspace{-.06in}\matrix{cc}{ \zeta_1& \frac{\Delta T\cos(\delta_1^*-\delta_2^*)}{c_{12}D_1} \\ \frac{\Delta T\cos(\delta_2^*-\delta_1^*)}{c_{12}D_2} & \zeta_2} \hspace{-.08in}\matrix{c}{ \hspace{-.08in} \Delta\delta_1(k) \hspace{-.08in} \\ \hspace{-.08in} \Delta\delta_2(k) \hspace{-.08in}}
\hspace{-.06in}+\hspace{-.06in}\matrix{cc}{\hspace{-.08in}u_1(k) \hspace{-.08in} \\ \hspace{-.08in}u_2(k)\hspace{-.08in}}\hspace{-.04in},
\end{split}
\end{equation*}
where 
$
\zeta_1=1\hspace{-.04in}-\hspace{-.04in}\Delta T(c_{12}^{-1}\cos(\delta_1^*-\delta_2^*) +c_1^{-1}\cos(\delta_1^*))/D_1$ and $
\zeta_2=1\hspace{-.04in}-\hspace{-.04in}\Delta T(c_{12}^{-1}\cos(\delta_2^*+\delta_1^*)- c_2^{-1}\cos(\delta_2^*))/D_2.
$
Consider the same variation of the local loads as in Example~\ref{example:1}. We get the discrete-time linear with stochastically-varying parameters 
$$
x(k+1)=Ax(k)+Bu(k) 
$$
where $x(k)=[\Delta\delta_1(k) \; \Delta\delta_2(k)]^\top$, $u(k)=[ u_1(k) \; u_2(k)]^\top,$ and
$$
B=\matrix{cccc}{1 & 0 \\ 0 & 1}, \hspace{.2in} A(k)=\matrix{cc}{-0.1635-0.2075\alpha_1(k)  &  0.7486 \\ 0.7486  & -0.1897-0.0877\alpha_2(k)}.
$$
where $\alpha_1(k)\sim\mathcal{N}(0,0.1)$ and $\alpha_2(k)\sim\mathcal{N}(0,0.3)$.
The goal is to optimize the performance criterion
$$
J=\mathbb{E}\left\{\sum_{k=0}^\infty x(k)^\top x(k)+u(k)^\top u(k)\right\}.
$$
Following Theorem~\ref{tho:4}, we can calculate the optimal controller with full model information as
\begin{equation*} \label{eqn:NOC1}
\begin{split}
u^{\scriptsize{\mbox{FMI}}}(k)=\matrix{cc}{
0.1166+0.1185\alpha_1(k)   & -0.4334-0.0027\alpha_2(k) \\
-0.4334-0.0064\alpha_1(k)  &  0.1317+ 0.0502\alpha_2(k)}x(k).
\end{split}
\end{equation*}
Furthermore, using Theorem~\ref{tho:2}, we can calculate the optimal controller with limited model information as
\begin{equation*} \label{eqn:NOC2}
\begin{split}
u^{\scriptsize{\mbox{LMI}}}(k)=\matrix{cc}{
    0.1166+0.1190\alpha_1(k) &  -0.4334 \\
   -0.4334 &   0.1317+0.0504\alpha_2(k)}x(k).
\end{split}
\end{equation*}
It is easy to see that
$$
r=\sup_{x_0\in\mathbb{R}^n}\frac{x_0^\T P^{\scriptsize{\mbox{LMI}}} x_0}
{x_0^\T P^{\scriptsize{\mbox{FMI}}} x_0}=1+1.2660\times 10^{-6}\leq 1+1/\epsilon^2=2,
$$
since $\epsilon=1$. In this example, the upper bound computed in Theorem~\ref{tho:upper} is not tight. \hfill\newqed
\end{example}

\begin{remark} \label{remark:deadbeat} Under Assumption~\ref{assmp:2}, when the variances of the plant model parameters tend to infinity, the optimal controller with limited model information (introduced in Theorem~\ref{tho:2}) approaches the deadbeat control law. The intuition behind this result is that when the model information of the other subsystems is inaccurate, the deadbeat control law (which decouples our subsystem from the rest of the plant) is the best controller to use. The presented approach balances in a natural way the use of statistical information about the plant parameters with precise knowledge of their realizations.
\end{remark}

\setcounter{example}{1}
\begin{example}[Cont'd] Let us consider the case where variances of the plant model parameters are very large. Hence, we assume $\alpha_1(k)\sim\mathcal{N}(0,1000)$ and $\alpha_2(k)\sim\mathcal{N}(0,3000)$. Now, the optimal controller with limited model information is given by
\begin{equation*}
\begin{split}
u^{\scriptsize{\mbox{LMI}}}(k)=\matrix{cc}{
    0.1635+ 0.2075 \alpha_1(k) &  -0.7485 \\
   -0.7485 &   0.1897+ 0.0877\alpha_2(k)}x(k),
\end{split}
\end{equation*}
which is practically equal to the deadbeat control law in Remark~\ref{remark:deadbeat:definition}. \hfill\newqed
\end{example}

\section{Conclusion} \label{sec:con}
We presented a statistical framework for the study of control design under limited model information. We found the best performance achievable by a limited model information control design method. We also studied the value of information in control design using the performance degradation ratio. Possible future work will focus on generalizing the results to discrete-time Markovian jump linear systems and to decentralized controllers.

\section{Acknowledgement}
The authors would like to thank C\'{e}dric~Langbort for valuable discussions and suggestions.

\bibliographystyle{ieeetr}
\bibliography{compile_new}

\end{document}